\documentclass{amsart}
  \usepackage{amscd,amssymb,epsfig, epsf}
   \usepackage{epic,eepic}
 \usepackage{epstopdf}

    \usepackage[all]{xy}
    \SelectTips{cm}{}
    \allowdisplaybreaks

    \numberwithin{equation}{subsection}

    \newtheorem{propo}{Proposition}[section]
    \newtheorem{corol}[propo]{Corollary}
    \newtheorem{theor}[propo]{Theorem}
    \newtheorem{lemma}[propo]{Lemma}

    \theoremstyle{definition}

    \theoremstyle{remark}

\newcommand{\ZZ}{\mathbb{Z}}

\newcommand{\RR}{\mathbb{R}}

 \newcommand{\G}{\mathcal{G}}
 \newcommand{\E}{\mathcal{D}}
 \newcommand{\F}{\mathcal{E}}
 \newcommand{\A}{\mathcal{A}}
  \newcommand{\B}{\mathcal{B}}
 \newcommand{\I}{\mathcal{I}}
\newcommand{\sal}{\mathcal{S}}
  \newcommand{\D}{\mathcal{D}}

\newcommand{\Ker}{\operatorname{Ker}}

\newcommand{\card}{\operatorname{card}}

\newcommand{\id}{\operatorname{id}}

\newcommand{\rel}{\operatorname{rel}}

\newcommand{\supp}{\operatorname{supp}}

\let\oldmarginpar\marginpar
\renewcommand\marginpar[1]{\oldmarginpar{\footnotesize #1}}

\begin{document}

    \title[{Matching groups   and gliding systems}]{Matching groups   and gliding systems}

    \author[Vladimir Turaev]{Vladimir Turaev}
    \address{
    Vladimir Turaev \newline
    \indent   Department of Mathematics \newline
    \indent  Indiana University \newline
    \indent Bloomington IN47405, USA\newline
    \indent $\mathtt{vturaev@yahoo.com}$}

                     \begin{abstract}  With every matching  in a   graph we associate a group   called the matching group. We study this group using the theory of nonpositively curved cubed complexes.   Our approach  is formulated  in terms of so-called gliding systems.
 \end{abstract}
                     \maketitle

   \section {Introduction}

 Consider a    graph ${\Gamma}$ without loops but possibly with multiple edges.   A  matching~$A$  in~$\Gamma$  is a  set of edges of~$\Gamma$  such that different edges in~$A$ have no common vertices. Matchings are extensively studied in graph theory usually with the view to  define  numerical invariants of graphs.  In this paper we study
    transformations of matchings  determined by even cycles. An even cycle  in~$\Gamma$ is an embedded circle  in~$\Gamma$  formed by an even number of edges. If   a matching~$A$     meets an even cycle~$s$ at every second edge of~$s$, then removing these edges from~$A$   and adding   instead all the other edges of~$s$  we
obtain a new matching  denoted  $sA$. We say that $sA$ is obtained
from $A$ by {\it gliding} along~$s$. The inverse transformation is the  gliding of~$sA$ along~$s$ which,  obviously, gives  back $A$. Composing  the glidings, we can  pass  back and forth between  matchings.     If two even cycles $s,t$   have no common   vertices and   a matching~$A$ meets both $s$ and $t$  at every second edge, then the compositions $A\mapsto sA \mapsto tsA$ and $A\mapsto tA \mapsto stA=tsA$ are considered as the same transformation.  For any matching~$A$ in~$\Gamma$,  the compositions of glidings carrying~$A$ to itself form a group $\pi_A=\pi_A(\Gamma) $  called the \emph{matching group}.
 Similar groups were first considered in \cite{STCR} in the context of domino tilings of planar regions.

 In the rest of the introduction, we focus on matching groups in   finite graphs.  We prove that  they are torsion-free,
residually nilpotent, residually finite, biorderable, biautomatic,
    have solvable word and conjugacy
problems, satisfy   the Tits alternative,     embed in $SL_n(\ZZ)$
for some~$n$, and embed  in    finitely generated
right-handed Artin groups.  Our main tool in the proof of these properties is an interpretation of the matching groups as the fundamental groups of  nonpositively curved cubed complexes. The universal coverings of such complexes are   Cartan-Alexandrov-Toponogov (0)-spaces in the sense of Gromov
    (CAT(0)-spaces).   All necessary definitions   from the theory of   cubed complexes are recalled in the paper.

Using much more elementary considerations, we    give  a presentation  of the matching group    by generators and relations as follows. The set of vertices of a finite graph~$\Gamma$ adjacent to the edges of a matching~$A$ in~$\Gamma$ is denoted $\partial A$.
We say that two matchings $A,B $ in~$\Gamma$ are {\it congruent}    if $\partial A=\partial B$. We explain that any tuple of matchings in~$\Gamma$   congruent to a given matching $A_0$  determines an element in $\pi_{A_0}$.  The group    $\pi_{A_0}$ is generated by  the elements
$\{x_{A, B}\}_{A, B }$ associated with the 2-tuples $A, B$ of matchings congruent to~$A_0$.  The defining relations:  $x_{A_0,A}=1$ for any $A $ congruent to~$ A_0$ and  $x_{A,C}=x_{A,B} \, x_{B,C}$ for any     matchings  $A,B,C $ congruent to~$ A_0$ such that  every vertex in $ \partial A_0$ is incident to an edge which belongs to at least two of the matchings  $A,B,C $.  As a consequence,  the group $\pi_{A_0}$  is finitely generated and its rank is smaller than or equal to  $M(M-1)/2$  where $M$ is the number of matchings in~$\Gamma$ congruent to~$A_0$ and distinct from~$A_0$.

  We    define  two families of natural   homomorphisms between   matching groups.
  First,  any subset $A'$ of a matching $A$ in~$\Gamma$  is itself a   matching  in~$\Gamma$. We define a canonical injection $ \pi_{A'}\hookrightarrow \pi_{A}$. Identifying  $\pi_{A'}$ with its image, one can    treat $\pi_{A'}$   as a subgroup of $\pi_A$.  Second,  any two congruent matchings $A,B$ in~$\Gamma$ may be related by glidings, and, as a consequence, their matching groups are isomorphic. We
 exhibit  a canonical  isomorphism $ \pi_A \approx \pi_B$.
    We  also relate the matching groups   to  the braid groups of   graphs. This allows us to derive  braids in graphs  from tuples of matchings.  We will briefly discuss a  generalization of  the matching groups   to hypergraphs.

    A special role in the theory of matchings is played by   perfect matchings also called dimer coverings. A matching in a graph is  perfect   if   every vertex of the graph  is incident to a (unique) edge of this matching.  Perfect matchings have been extensively studied   in connection with   exactly solvable models of statistical mechanics and  with path algebras, see \cite{Bo},   \cite{Ken2} and  references therein.   The matching groups associated with perfect matchings are called \emph{dimer groups}.
Since  all  perfect matchings in a finite graph   are congruent, their    dimer  groups are isomorphic.
The resulting isomorphism class of   groups   is an invariant of the graph.



    The study of  glidings    suggests a more general   framework of  gliding systems in groups.  A {\it gliding
system} in a group $G$ consists of certain elements of~$G$ called  {\it  glides} and a    relation on the set of glides called {\it independence} satisfying a few   axioms.   Given a gliding system in $G$ and a  set $\E \subset G$, we construct a cubed complex~$X_\E$ called the {\it glide complex}.       The fundamental groups of the components of~$X_\E$ are  the {\it glide groups}. We formulate  conditions   ensuring that $X_\E$ is nonpositively curved. One can view gliding systems as  devices producing   nonpositively curved complexes and interesting groups.
 The matching groups  and, in particular, the dimer groups   are   instances  of   glide groups for appropriate~$G$ and~$\E$.

The paper is organized as follows. In Section~\ref{Poccacm} we recall  the basics on cubed complexes and cubic maps. The next three sections
 deal with glidings: we define the gliding systems (Section~\ref{sect-prelim}), construct the glide complexes (Section~\ref{ From glides to cubed   complexes}), and  study   natural   maps between the glide groups (Section~\ref{incmaps}).   Next, we introduce    dimer groups (Section~\ref{dimergroups}),   compute them via  generators and relations (Section~\ref{redu}), and   define and study the matching groups     (Section~\ref{gencasem}). In Section~\ref{extension-----} we consider connections with braid groups. In Section~\ref{lloccgccgcgl}  we interpret the dimer complex  in terms of graph labelings.   In Section~\ref{extension}  we   discuss the matching groups of hypergraphs. In the appendix we  examine the typing homomorphisms of the matching groups.

The author would like to thank  M. Ciucu for several stimulating discussions. This work   was partially supported by the NSF
  grant  DMS-1202335.

\section{Preliminaries on cubed complexes and cubical maps}\label{Poccacm}

We discuss the basics of the theory of
  cubed complexes  and cubical maps, see
\cite{BH}, Chapters I.7 and II.5 for more details.


\subsection{Cubed complexes}\label{Cubed complexes} Set $I=[0,1]$. A {\it cubed complex}  is a
CW-complex $X$   such that each (closed) $k$-cell of $X$ with $k\geq
0$ is a continuous map   from the $k$-dimensional cube $I^k$ to $X$
whose restriction  to the interior of $I^k$ is injective and whose
restriction  to each $(k-1)$-face of $I^k$ is an isometry of that
face onto $I^{k-1}$ composed with a $(k-1)$-cell $I^{k-1}\to X$ of
$X$.  The $k$-cells $I^k \to X$   are not required to be injective.
The $k$-skeleton  $X^{(k)}$  of $X$ is the union of the images of all
 cells of dimension  $ \leq k$.

 For example, the cube $I^k$ together with all its faces is a cubed complex.
So is the $k$-dimensional torus obtained by identifying opposite faces of $I^k$.

The {\it link} $LK(A)=LK(A;X)$ of a 0-cell $A $ of a cubed complex
$X$ is the space of all directions at $A$. Each triple ($k\geq 1$, a
vertex $a$ of $I^k$,   a $k$-cell $\alpha:I^k \to X $ of $X$
carrying $a$ to $A$) determines a  $(k-1)$-dimensional simplex in
$LK(A)$ in the obvious way. The faces of this simplex are determined
by the restrictions of $\alpha$ to the faces of $I^k$
containing~$a$. The simplices corresponding to all triples
$(k,a,\alpha)$ cover $LK(A)$ but may not form a simplicial complex.
We say, following \cite{HW}, that the cubed complex $X$ is {\it
simple} if the links of all $A\in X^{(0)}$ are    simplicial complexes,
i.e., all   simplices  in $LK(A)$    are embedded and
the intersection of any two simplices in $LK(A)$  is a common face.


A {\it flag complex} is a simplicial complex such that any finite
collection of pairwise adjacent vertices spans a simplex.  A cubed
complex is {\it nonpositively curved} if it is simple and the link
of each 0-cell is a flag complex. A theorem of M. Gromov asserts
that the universal covering of a  connected finite-dimensional
nonpositively curved cubed complex    is   a CAT(0)-space.
Since
CAT(0)-spaces are contractible, all higher homotopy groups of  such a complex~$X$
vanish while the fundamental group $\pi=\pi_1(X)$ is torsion-free. This group
     satisfies a strong form of the Tits alternative: each
subgroup of $\pi $ contains a rank 2 free subgroup or    virtually is a
finitely generated abelian group, see \cite{SW}. Also,  $\pi$ does
not have Kazhdan's property (T), see \cite{NR1}. If $X$ is  compact,
then $\pi$ has solvable word and conjugacy problems and is
biautomatic, see \cite{NR2}.

\subsection{Cubical maps}\label{llocal}   A {\it cubical map} between simple cubed complexes $X$ and $Y$
  is a continuous map $X\to Y$ whose composition
with any $k$-cell $I^k\to X$ of $X$ expands  as   the composition of a
self-isometry of $I^k$ with a $k$-cell $I^k\to Y$ of~$Y$ for all
$k\geq 0$. For  any $A\in X^{(0)}$, a cubical   map $f: X\to Y$ induces a
simplicial
  map $f_A:LK(A)\to LK(f(A))$.   A cubical map
$f: X\to Y$   is a {\it local isometry} if  for all $A\in X^{(0)}$, the map~$f_A$
is an embedding and its image is  a full subcomplex of $LK(f(A))$. The latter condition means  that  every simplex  of $LK(f(A))$ whose vertices are images under $f_A$ of vertices of  $LK(A)$ is itself the image under $f_A$ of a simplex of  $LK(A)$.

 \begin{lemma}\label{Wise} \cite[Theorem 1.2]{CW} If $f:X\to Y$ is a local isometry of nonpositively curved finite-dimensional
 cubed
 complexes, then the induced group homomorphism $f_*: \pi_1(X,A)\to \pi_1(Y,
 f(A))$ is injective for all $A\in X$.
   \end{lemma}

 \section{Glides}\label{sect-prelim}

 \subsection{Gliding systems} \label{BU} For a   subset $\G$ of a group $ G$, we set $\overline \G=\{s^{-1}\, \vert \, s\in \G \}  \subset G$. A  {\it  gliding system} in a group $G$ is  a pair of sets
 $$\G\subset G \setminus \{1\} , \quad \I\subset \G\times \G$$ such that $\overline \G=\G$ and  for any $s,{t}\in \G$ with
  $(s , {t}) \in \I$, we have     $(s^{-1} , {t}),
  ({t},s)\in
 \I$, and $s {t}={t}s \neq 1$.
The elements of $\G$ are called
  {\it glides}.   The inverse of a glide is a glide while the unit   $1\in G$ is never a glide. The set $\I$ determines   a symmetric  relation on~$\G$
    called {\it
 independence}.  Thus, two glides $s,t$ are independent if and only if $(s,t)\in \I$.  For any independent glides   $s,t$, we have  $st=ts$, $  s\neq t^{\pm 1}  $, and $(s^{\varepsilon} ,   {t}^{\mu}), ({t}^{\mu} ,s^{\varepsilon}) \in \I$ for all $\varepsilon, \mu =\pm 1$. In particular,  a glide is never independent from itself or
its  inverse.  We do not require the glides   to generate~$G$ as a group.

Let $(\G  , \I )$ be  a gliding system  in a group~$G$.  For   $s\in \G$  and    $A \in G$,  we say that $sA\in G$ is
 obtained from $A$ by {\it (left) gliding along}~$s$. One can
 similarly consider right glidings but we do not need them.

A   subset      $S $ of $ \G$ is {\it  pre-cubic} if it is finite and $(s,t)\in \I$ for any distinct $s,t\in S$.  Since  independent
glides commute, such a set  $S$ determines
an element $[S]=\prod_{s\in S} s$ of~$ G$. In particular, the empty set   $\emptyset \subset \G$ is pre-cubic and
$[\emptyset]=1$.

  A  set  $S\subset \G$ is  {\it cubic} if it is    pre-cubic   and for any distinct  subsets
$T_1,T_2 $ of~$  S$, we have  $[T_1]\neq [T_2]$.   In particular,
 $[T]\neq [\emptyset]=1$  for any non-empty
$T\subset S$. Examples of cubic sets of glides: $\emptyset$;
 a set consisting of a single glide;
a   set consisting of two independent glides.  It is clear that any subset  of a cubic set of glides is cubic.

  \begin{lemma}\label{preglidi-}  For each subset $T$ of a (pre-)cubic set of glides $S \subset G$, the set  of glides
  $S_T=(S\setminus T)\cup \overline T$ is    (pre-)cubic where $\overline T=\{t^{-1}\, \vert \, t\in T \}  \subset G$.\end{lemma}

\begin{proof} That $S_T$ is finite and consists of pairwise independent glides follows from the definitions. We need   to prove that if $S$ is cubic, then so it $S_T$. It suffices to  show that $[T_1]=[T_2] \Longrightarrow T_1=T_2$ for any subsets $T_1, T_2 $ of $ S_T$. For $i=1,2$,   put $U_i=(S\setminus T) \cap T_i$ and $V_i=T\cap  \overline T_i$. Clearly, $ \overline  V_i=  \overline  T\cap  T_i$ and $ U_i \cup \overline V_i=T_i$. The independence of the elements of $S$ implies that  the sets $S\setminus T$ and $\overline T$ are disjoint. Therefore their subsets $U_i $ and $\overline V_i$ are disjoint and
$$[T_i]= [U_i \cup \overline V_i]=[U_i] [\overline V_i] = [ U_i] [V_i]^{-1}  . $$  The assumption $[T_1]=[T_2]$ implies that $[ U_1] [V_1]^{-1}=[ U_2] [V_2]^{-1}$. The sets  $V_1, V_2$ are subsets of $T\subset S$ and therefore the elements $[V_1], [V_2]$ of~$G$ commute. Hence
$[U_1] [V_2]= [ U_2]   [V_1]$. The equalities $U_1 \cap V_2=U_2\cap V_1 =\emptyset$ imply that
$$ [ U_1 \cup V_2] = [U_1] [V_2]= [ U_2]   [V_1]=[U_2\cup V_1] .$$
Since     $U_1 \cup V_2$ and $U_2\cup V_1$ are subsets  of the cubic set $S $, they must be equal.
So,
$$U_1=(U_1 \cup V_2) \cap (S\setminus T) = (U_2 \cup V_1) \cap (S\setminus T)=U_2$$
and
$$V_1= (U_2 \cup V_1) \cap T= (U_1\cup V_2) \cap T = V_2.$$
We conclude that $T_1=U_1\cup \overline V_1 =U_2\cup \overline V_2 =T_2$.
\end{proof}

\subsection{Examples}\label{exam} 1.   For any group $G$, the following pair is a   gliding
system:  $\G= G \setminus \{1\}$  and $\I$ is the set of all pairs
$(s,t)\in \G\times \G$ such that $s\neq t^{\pm 1}$ and $st=ts$.

2. For a   group~$G$ and a  set $\G \subset G \setminus \{1\}$ such that $\overline \G=\G$, the pair
$(\G , \I=\emptyset)$ is a   gliding
system. The cubic subsets of $\G$ are the empty set and the 1-element  subsets.


3. Let $G $ be a free abelian group   with   free commuting generators $\{g_i\}_i$. Then
$$\G=\{ g_i^{\pm 1} \}_{i}, \quad \I=\{(g_i^\varepsilon,g_j^\mu)
\, \vert \, \varepsilon=\pm 1,
\mu =\pm 1,   i\neq j\}$$ is a
  gliding system in $G$. A cubic subset of $\G$ consists of a finite number of    $ g_i $ and a finite number of   $ g_j^{-1} $ with $i\neq j$.

4. A   generalization of the previous example is provided by the
theory of right-angled Artin groups (see \cite{Ch} for an
exposition). A right-angled Artin group is a group allowing a
presentation  by generators and relations in which all relators are
commutators of the generators. Any graph ${\Gamma}$ with the set of
vertices $V$ determines a right-angled Artin group $G=G({\Gamma})$
with generators $\{g_s\}_{s\in V}$   and relations $g_s g_t= g_t
g_s$   whenever   $s, t \in V$ are connected by an
edge in~${\Gamma}$ (we write then $s\leftrightarrow t$). Abelianizing
$G$ we obtain that $g_s\neq g_t^{\pm 1}$ for $s\neq t$. The pair
$$ \G=\{ g_s^{\pm 1} \}_{s\in V}, \quad
\I=\{(g_s^\varepsilon, g_t^\mu)\, \vert \, \varepsilon , \mu
=\pm 1, s,t \in V, s\neq t, s\leftrightarrow t
 \}$$  is a
gliding system in $G$.

5.  Let $E$ be  a set  and     $ G= 2^E$   be the power set of $E$ consisting of all
subsets of $E$. We define multiplication  in $G$ by  $ A  B = (A\cup
B)\setminus (A\cap B)$  for   $A,B\subset E$. This   turns $G$
into an abelian   group with unit $1=\emptyset$, the {\it power group of
  $E$}.  Clearly, $A^{-1}=A$ for all $A\in G$. Pick any set
$\G\subset G\setminus \{1\}$ and declare
  elements of~$\G$   independent when  they are disjoint as
subsets of~$E$. This gives  a     gliding system in~$G$. A cubic subset of $\G$ is just a finite collection of pairwise disjoint non-empty subsets of~$E$.

6.  Let $E$ be a set, $H$ be a multiplicative group, and  $G=H^E$ be the group of
all maps $E\to H$ with pointwise multiplication. The {\it support} of a map $f:E\to H$ is the set $\supp(f)=
  f^{-1} (H\setminus \{1\})\subset E$. Pick a   set $\G\subset G\setminus \{1\}$ invariant under inversion and  declare
  elements of $\G$   independent if
   their supports are disjoint. This gives
  a
gliding system in $G$. When $H$ is a cyclic group of order~$2$, we
recover Example 5 via the group isomorphism  $H^E\simeq 2^E$ carrying a map
$E\to H$ to its support.

\subsection{Regular gliding systems}\label{regg}  A gliding system   is {\it regular} if all pre-cubic sets of  glides in this system  are cubic.  We will be mainly  interested in regular   gliding systems. The gliding system  in Example~\ref{exam}.1 may be non-regular while
those  in Examples~\ref{exam}.2--6 are  regular.
The regularity in Examples~\ref{exam}.5 and~\ref{exam}.6  is a  consequence  of the following     lemma.

\begin{lemma}\label{glidi-}  Let    $E$ be a set  and $H$ be a  group. Consider a  gliding
system in the group $G=H^E$ such that  the supports of any two  independent glides
  are disjoint  (as subsets of $E$).  Then this gliding system
is regular.
\end{lemma}

\begin{proof}     Consider an arbitrary  pre-cubic set of glides  $S\subset G$.  Any subset $T$ of $ S$ is also pre-cubic. Since the supports of   independent glides
  are disjoint,    $$\supp([T])= \amalg_{s\in T}  \supp (s)\subset E.$$      Consider two distinct subsets $T_1,T_2$   of $S$. Assume for
concreteness that    $T_1 \setminus T_2 \neq \emptyset$. Pick   $t\in T_1 \setminus T_2$. Then $\supp(t) \subset \supp([T_1])$ and  $\supp(t)  \cap \supp([T_2]) =\emptyset$. Since~$t$ is a glide, $\supp(t)\neq \emptyset$. Hence $[T_1]\neq [T_2]$. This proves that the set $S$ is cubic.
\end{proof}

\section{ Glide    complexes and glide groups}\label{ From glides to cubed   complexes}

In  this section, $G$ is a group equipped with a gliding system. We define a cubed complex~$
X_G$, the {\it glide complex},  and study certain  cubed  subcomplexes of $X_G$.

\subsection{The glide  complex}\label{cubecomplexX}   A {\it based cube} in $G$ is a pair  ($A\in G$, a cubic set of
 glides $S\subset G$). The integer $k=\card(S)\geq 0$ is the {\it dimension} of the based cube
$(A,S)$. It follows from the definition of a cubic set of glides that  the set $\{[T]A\}_{T\subset S} \subset G$ has~$2^k$
elements; these elements are called  the {\it vertices}
of the based cube $(A,S)$.

  Two based cubes $(A,S)$ and $(A', S')$ are   {\it equivalent} if there is a set $T\subset S$ such that
$A'=[T]A$ and $S'=S_T$ where $S_T$ is   defined in Lemma~\ref{preglidi-}. This is
indeed an equivalence relation on the set of based cubes. Each $k$-dimensional based cube is equivalent to $2^k$
based cubes (including itself). The equivalence classes of
$k$-dimensional based cubes are called {\it $k$-dimensional  cubes} or {\it $k$-cubes}
in $G$. Since equivalent based cubes have the same vertices, we may
speak of the vertices of a $k$-cube. The $0$-cubes in $G$ are just the elements of $G$.

A cube $Q$ in $G$ is a {\it face} of a cube $Q'$  in $G$  if $Q,Q'$ may be
represented by based cubes $(A,S)$, $(A', S')$, respectively, such
that $A=A'$ and $S\subset S'$. Note that  in the role of $A$ one
may take an arbitrary vertex of $Q$.

The  glide complex  $X_G$ is the cubed complex obtained by taking a
copy of $I^{k}$ for each    $k$-cube    in $G$ with
$k\geq 0$ and gluing these copies via
 identifications determined by inclusions of     cubes   into bigger     cubes   as their faces.
Here is a more precise definition.   A point of $X_G$ is represented
by a triple $(A,S,x\in I^S)$ where $(A,S)$ is a  based cube in $G$ and  $I^S$
is  the set of all maps $ S \to I $ viewed as the product of
copies of~$I$ numerated by elements of $S$ and provided with the
product topology. For fixed $(A, S)$,    the  triples $(A,S,x)$ form the
geometric cube $I^S$.  We take a
disjoint union of these cubes over all   $(A,S )$ and factorize it
by the equivalence relation generated by the  relation
$(A,S,x) \sim (A',S',x')$ when $$ A=A',\,\, S\subset S', \,\, x=x'\vert_{S},\,\,
x'(S'\setminus S)=0$$ or there is a set $T\subset S$ such that
$$A'=[T]A, \,\, S'=S_T , \,\, x'\vert_{S'\setminus T'}=x\vert_{S\setminus T} \,\,   {\text{and}}\,\, x(t)+ x'(t^{-1})=1  \,\,  {\text{for all}}\,\,  {t\in T}. $$The  quotient space $X_G$
is a cubed space in the obvious way.

 \begin{lemma}\label{newlemmw-}  All cubes in~$G$ are embedded in $X_G$. \end{lemma}

\begin{proof} This   follows from the fact that different faces of a $k$-cube   have different sets of vertices and therefore are never glued to each other under our identifications.   \end{proof}

\subsection{Subcomplexes of $X_G$}\label{cubecomplexX+}   Any
 set $\E\subset G$  determines  a  cubed complex $X_\E \subset X_G$ formed
by the cubes in $G$ whose  all vertices belong to~$\E$.  Such cubes
are called {\it cubes in $\E$}, and   $X_\E$ is called the {\it
glide complex} of $\E$.
For $A\in
\E$, the fundamental group $\pi_1(X_\E, A)$ is called  the {\it glide group   of $\E$ at $A$}. The  elements of
$\E$ related by glidings  in $\E$  belong to the same component of~$X_\E$ and give rise to isomorphic glide groups.


 \begin{lemma}\label{simple}    For any set $\E\subset G$, the  cubed complex $X_\E$ is simple in the sense of Section~\ref{Cubed complexes}.
   \end{lemma}

\begin{proof} A neighborhood of $A\in G$ in $X_G$ can be obtained by taking all triples $(A,S,x)$, where $S$ is a  cubic set of glides and $x(S)\subset [0,1/2)$, and  identifying two such triples $(A,S_1,x_1)$, $ (A,S_2,x_2)$ whenever $  x_1=x_2$ on $S_1\cap S_2$ and $
x_1(S_1\setminus S_2)=x_2(S_2\setminus S_1)=0$. Therefore, the link $LK_G(A)$ of $A $ in $X_G$ is the simplicial complex whose vertices are    the glides    and whose simplices are  the cubic sets of  glides. The link $LK_\E(A)$ of $A \in \E$ in $X_\E$ is the subcomplex of $LK_G(A)$ formed by the glides $s\in G$ such that $sA\in \E$ and the cubic sets $S\subset G$ such that $[T]A\in \E$ for all $T\subset S$.
\end{proof}

For $A\in \E $, we now reformulate the flag condition on
the link $LK_\E(A)$ of $A $ in $X_\E$ in terms of glides. Observe that finite  sets of pairwise
 adjacent vertices in $LK_\E(A)$ bijectively correspond to
   pre-cubic sets of glides $S \subset G$  such that

$(\ast)$  $sA\in \E$ for all $s\in S$  and $s{t} A\in \E$ for all
distinct $s, {t}\in S$.

\begin{lemma}\label{simple+}  The link $LK_\E(A)$ of $A \in \E$ in $X_\E$ is a flag complex if and only if any pre-cubic set  of glides $S \subset G$  satisfying $(\ast)$ is cubic and   $[S]A\in \E$.
   \end{lemma}

This lemma follows directly from the definitions. One should use the
obvious fact  that if a pre-cubic set of glides satisfies
$(\ast)$ then so do all its subsets.

 We now formulate combinatorial conditions on  a set $\E\subset G$  necessary and sufficient for
 $X_\E$ to be nonpositively curved  in the sense of Section~\ref{Cubed complexes}. We say that~$\E$ is  {\it regular}
if    for every $A\in \E$, all     pre-cubic sets
 of glides  $S \subset G$  satisfying $(\ast)$ are cubic. We say that~$\E$ satisfies the  {\it   cube condition} if for any
 $A\in \E$ and  any  pairwise
independent glides $s_1, s_2, s_3\in G$ such that  $  s_1A, s_2 A, s_3
A, s_1 s_2 A, s_1 s_3 A, s_2 s_3A\in \E$, we necessarily have  $s_1 s_2 s_3 A\in
\E$.
This condition  may be reformulated by saying that if seven vertices of a 3-cube in~$G$ belong to~$\E$, then so does the  eighth vertex.




 \begin{theor}\label{simpleNEW}  The  cubed complex $X_\E$ of a  set $\E\subset G$  is nonpositively curved if and only if  $\E$   is regular and satisfies the  cube condition.
\end{theor}


 \begin{proof}  Lemmas \ref{simple}  and \ref{simple+} imply that  $X_\E$ is
nonpositively curved if and only if for all $A\in \E$,
 each       pre-cubic set   of
glides $S\subset G$    satisfying $(\ast)$ is cubic and    $[S]A\in \E$. We
must only show that the inclusion $[S]A\in \E$ can be replaced with
the   cube condition.
 One direction is obvious: if $A, s_1, s_2, s_3\in G$ satisfy the assumptions of  the cube condition, then
the set  $S=\{s_1,s_2, s_3\}$  satisfies $(\ast)$ and so
$s_1s_2s_3 A=[S]A\in
  \E$.
 Conversely, suppose that $\E$ is regular and meets the cube condition.  We must  show
that $[S]A\in \E$ for any $A\in \E$ and any pre-cubic set
 of  glides $S \subset G$  satisfying $(\ast)$.   We proceed by induction on
$k=\card(S)$. For $k=0$, the claim follows from the inclusion $A\in
\E$. For $k=1,2$, the claim follows from $(\ast)$. If $k\geq 3$,
then the induction assumption guarantees that $[T]A\in \E$ for any
proper subset $T\subset S$. If $S=\{s_1,...,s_k\}$, then applying
the cube condition   to $s_1,s_2,s_3$ and the element $s_4 \cdots
s_{k}A $ of $\E$, we obtain that $[S]A\in \E$.
 \end{proof}

By Section \ref{Cubed complexes}, if $X_\E$ is nonpositively curved,
then the universal covering   of every finite-dimensional component
of $X_\E$ is a CAT(0)-space. The component itself is then an
Eilenberg-MacLane space of type $K(\pi,1)$  where $\pi$ is the
corresponding glide group. Note that  $\dim X_\E$ is the maximal dimension of a cube in~$\E$. In particular, if   $\E$ is a finite set, then
$ X_\E  $ is a finite-dimensional complex.

If the  gliding system in~$G$ is  regular, then  all
subsets of~$G$ are regular.  As a consequence, we obtain  the following corollary.

  \begin{corol}\label{simpleNEWBIS}  For a  group $G$ with a regular gliding system,
the glide complex $X_\E$ of    a  set $\E\subset G$  is nonpositively curved if and only if $\E$ satisfies the cube condition. In particular, the glide complex $X_G$ of $\E=G$  is nonpositively curved.
\end{corol}

Corollary~\ref{simpleNEWBIS}   applies, in particular,  to the regular gliding systems described in Examples~\ref{exam}.2--6.

 \subsection{Remarks}\label{examsec2}  We make a few miscellaneous remarks on the glide complexes.

 1. The group~$G$ acts on $X_G$ on the right   by   $(A,S,x)g=(Ag,S,x)$ for   $g\in G$. This action preserves the cubed structure   and is free and transitive on
  $X^{(0)}_G=G$.


 2.   In Example~\ref{exam}.2,  $X_G $ is the
  graph with the set of vertices $G$, two vertices $A, B\in G$ being connected by a (single) edge if and only if $AB^{-1}\in \G$.   All subsets of $G$ satisfy the
 cube condition because $G$ has no independent glides. The corresponding glide groups are free.


 3.  In Example~\ref{exam}.3, if  the rank, $n$, of $G$ is finite, then  $X_G=\RR^n$ with the standard action of~$\ZZ^n$.


4.  In Example~\ref{exam}.4, $X_G$ is simply-connected,   the   action of $G$ on $X_G$
 is free, and  the projection $ X_G\to X_G/G$ is the universal covering
of
 $X_{G}/G$. Composing the cells  of $ X_G$ with this projection we turn $X_G/G$ into a cubed complex  called the
  Salvetti complex, see \cite{Ch}.  This complex   has only one
0-cell; its link    is isomorphic to the link of any vertex
of $X_G$ and  is a flag complex. This recovers the well known fact
that the Salvetti complex is nonpositively curved. Clearly, $\dim X_{G}
=\dim (X_{G}/G)$ is the maximal
  number of  vertices of a  complete subgraph   of the graph~${\Gamma}$.

5.  In Examples~\ref{exam}.5 and \ref{exam}.6,   if    $E$ is finite, then $X_G$ is finite dimensional.



 \section{Homomorphisms of glide groups}\label{incmaps}

 We discuss two families of homomorphisms of glide groups: the inclusion homomorphisms  and the typing homomorphisms.

   \subsection{Inclusion homomorphisms}\label{incl++} Let $G$ be a group   with a
   gliding system. Consider any sets $\F\subset \E\subset G$ and the associated cubed complexes
   $ X_{\F} \subset  X_{\E} \subset  X_{G}$.      We say that  $\F  $
   satisfies the {\it square condition   $\rel \E$}, if for any $A\in \F$ and
   any independent glides $s,t\in G$ such that $sA, tA\in \F, stA \in
   \E$, we necessarily have $stA\in \F$.  This condition may be reformulated by saying that if three vertices of a square (a 2-cube) in $X_\E$ belong to $X_\F$, then so does the  fourth vertex. For example, the
   intersection of $\E$ with any subgroup of~$G$  satisfies the  square condition   $\rel
   \E$.

 \begin{theor}\label{ecu} Let $\E\subset G$ be a regular set    satisfying  the cube condition  and such that $\dim X_\E<\infty$. Let $\F$ be a subset of $\E$
 satisfying the square condition $\rel \E$. Then  the
 inclusion homomorphism $\pi_1(X_{\F}, A) \to \pi_1(X_\E, A)$ is
 injective for all $A\in \F$.
\end{theor}

\begin{proof}   Since $\E$ is regular,
   so is $\F\subset \E$. The cube condition on $\E$ and the square condition $\rel \E$ on
   $\F$     imply that $\F$ satisifes the cube
   condition. By Theorem~\ref{simpleNEW},   $X_\E$ and $X_{\F}$ are nonpositively
   curved. By
   assumption, the cubed complex $X_\E$ is finite-dimensional and so
   is its subcomplex $X_{\F}$. We claim that the inclusion $X_{\F}\hookrightarrow
   X_{\E}$ is a local isometry. Together with Lemma~\ref{Wise} this
   will imply the   theorem.

To prove our claim,   pick any $A\in \F$ and consider the simplicial
complexes $L'=LK(A, X_{\F})$, $L=LK(A, X_{\E})$.  The inclusion
$X_{\F}\hookrightarrow
   X_{\E}$ induces an embedding $L'\hookrightarrow L$, and we need only to verify that the image of $L'$ is  a full
subcomplex of~$L$.  Since   $L'$ is a  flag simplicial
complex, it suffices to verify that any  vertices of~$L'$
adjacent in~$L $ are  adjacent in~$L'$.
This follows
   from the square condition  on~$\F$.
 \end{proof}

The set $\E\subset G$ satisfies the {\it square condition} if it satisfies the square
condition $\rel G$. In other words,  $\E $ satisfies the   square condition if for any $A\in \E$ and
   any independent glides $s,t\in G$ with $sA, tA\in \E $, we necessarily have $stA\in \E$.

\begin{corol}\label{simpleNEW++-}  If   the gliding system in~$G$ is regular and $\dim X_G<\infty$, then for every    set $\E\subset G$ satisfying the square
condition  and every   $A \in  \E$,   the
 inclusion homomorphism $\pi_1(X_{\E}, A) \to \pi_1(X_G, A)$ is
 injective.
\end{corol}


   \subsection{Typing homomorphisms}\label{prelimArtin++}   A group $G$ carrying  a
gliding system $(\G,\I)$ determines a right-angled Artin group
$\A=\A(G)$ with generators $\{g_s\}_{s \in \G }$ and    relations $g_s
g_t =g_t g_s$ where $(s, t)$ runs over $ \I $.  We now relate $\A$  to the glide groups.

First,  we introduce a notion of an orientation on
a  set $\E
\subset G$. An {\it orientation} on~$\E$ is  a
choice of direction on each 1-cell of the glide complex $X_\E$ such
that    the opposite sides of any  2-cell of $X_\E$ (a square) point
towards each other on the boundary of the square. In other words,  for any
based square  $(A,\{s,t\})$ with $A, sA, tA, stA\in \E$, the 1-cells
connecting $A$ to $sA$ and $tA$ to $stA$ are  either both directed towards
$sA, stA$ or both directed towards $A,tA$ (and similarly with $s$, $t$
exchanged). A set $\E \subset G$ is {\it orientable} if it allows an
orientation and is {\it oriented} if it has a
distinguished orientation. An orientation of $\E$ induces an
orientation of any  subset  $\F\subset  \E$ via the inclusion
$X_{\F}\subset X_\E$. Therefore, all subsets of an orientable set
are orientable.  These definitions   apply, in particular,  to $\E=G$. Examples of oriented sets will be given in Section~\ref{theglidingsystemcycles}.


Consider an oriented set $\E\subset G$. Each 1-cell $e$ of $X_\E$ is oriented and so leads from a vertex $A\in \E$ to a vertex $B\in \E$. We set
$\vert e\vert= BA^{-1}\in G$. It follows from the definition of $X_\E$ that $\vert e\vert$ is a glide in $G$. Consider now a  path $\alpha$
in the 1-skeleton of $X_\E$    formed by $n\geq 0$ consecutive
 1-cells $e_1,...,e_n$. The path $\alpha$ determines an orientation of $e_1,...,e_n$ so that the terminal endpoint of
$e_k$ is the initial endpoint of $e_{k+1}$ for $k=1, ..., n-1$. This
orientation of   $e_k$   may coincide or not with that
given by the orientation of $\E$. We set $\nu_k=+1$ or $\nu_k=-1$,
respectively. Set
\begin{equation}\label{pathee} \mu(\alpha)=  g_{\vert e_1\vert}^{\nu_1} \,  g_{\vert e_2\vert}^{\nu_2} \cdots g_{\vert e_n\vert}^{\nu_n} \in \A=\A(G)  .\end{equation}
 It is clear that   $\mu(\alpha)$ is preserved under  inserting in
the sequence $e_1,...,e_n$ two opposite 1-cells  or four
1-cells forming the boundary of a 2-cell. Therefore   $\mu(\alpha)$ is
preserved under homotopies of $\alpha$ in $X_\E$ relative to the
endpoints. Applying $\mu$ to loops
based at $A\in \E$, we obtain a   homomorphism $\mu_A: \pi_1(X_\E,
A)\to \A $. Following the terminology of \cite{HW}, we call $\mu_A$ the {\it typing homomorphism}.

\begin{theor}\label{simpleNEW++}  If there is  an upper bound
on the number
 of pairwise independent glides in $G$, then for any   oriented regular    set $\E\subset G$  satisfying  the square condition and for any $A\in \E$, the typing homomorphism  $\mu_A: \pi_1(X_\E,
A)\to \A $  is   injective.
\end{theor}



\begin{proof} 
By Example \ref{exam}.4, the group $\A $  carries a
gliding system with glides    $\{g_s^{\pm 1}\}_{s\in \G}$. Two glides
$g_s^{\pm 1}, g_t^{\pm 1} \in \A$  are independent if
and only if $s\neq t$ and $(s,t)\in \I$.  Consider the
associated cubed complex $X=X_{\A}$     and the Salvetti complex $Y=X /\A$.  Recall that $\A=\pi_1(Y, \ast)$ where $\ast$ is the unique 0-cell of $Y$.     By Section~\ref{examsec2}, both
$X $ and $Y$ are nonpositively curved. The assumptions of the
theorem imply that the cubed complex $  X_\E $ is nonpositively
curved, and the spaces $X $, $Y$, $ X_\E $   are finite-dimensional. We claim
that the homomorphism $   \mu_A: \pi_1(X_\E, A)\to
\A $ is induced by a local isometry $X_\E\to Y$.
By Lemma~\ref{Wise},  this will imply the theorem.

Since the gliding system in $\A$ is regular, the points of $X $ are represented by  triples $(A\in
\A, \sal , x\in I^\sal)$ where~$\sal$ is a pre-cubic set of glides in~$\A$  that is     $S=\{g_s^{\varepsilon_s}\}_s$ where~$s$ runs over a finite set of   independent glides in~$G$ and
$\varepsilon_s\in \{\pm 1\}$. The space~$X $ is obtained by   factorizing  the set of
  such triples
  by the equivalence relation defined in Section~\ref{cubecomplexX}. The space~$Y$ is obtained from~$X $ by forgetting the first term,~$A$, of the triple.
A point of $Y $ is represented by a  pair   (a pre-cubic set of glides $\sal\subset \A$, $x\in I^\sal$). The space $Y $  is obtained by   factorizing  the set of
  such pairs by the equivalence relation generated by the
following relation: $( \sal,x) \sim ( \sal',x')$ when $
\sal\subset \sal', x=x'\vert_{\sal}, x'(\sal'\setminus \sal)=0$ or
there is   $T\subset \sal$ such that $  \sal'=\sal_T $,
$x'=x$ on $\sal\setminus T$, and $x'(t^{-1})=1-x(t)$ for all $ {t\in
T}$.

We now construct a cubical map $f:X_\E\to Y$ carrying
each 1-cell~$e$     of  $X_\E$   onto the 1-cell of $Y$ determined by
$g_{\vert e\vert} $ where $\vert e\vert \in G$ is the glide determined by the distinguished orientation of $e$. Here is a precise definition of $f$. A point $a\in X_\E$ is
represented by a triple $(A\in \E, S ,x\in I^S)$ where~$S $ is a cubic set of glides in~$G$ such that $[T]A\in \E$ for
all $T\subset S$. For  $s\in S$, set $\vert s \vert =\vert e_s
\vert$ where~$e_s$ is the 1-cell of~$X_\E$ connecting $A$ and $sA$.
By definition, $\vert s \vert =s^{\varepsilon_s}$ where
$\varepsilon_s=+1$ if $e_s$  is oriented towards $ sA$ and
$\varepsilon_s=-1$ otherwise. Let $f(a)\in Y$ be
represented by the pair  $(  \sal=\{g_{\vert s
\vert}^{\varepsilon_s}\}_{s\in S}, y\in I^\sal)$ where $y(g_{\vert
s \vert}^{\varepsilon_s})=x(s)$ for all $s\in S$. This yields a well-defined cubical map $f:X_\E\to Y$ inducing
$\mu_A$ in $\pi_1$.

It remains to show that $f$ is a local isometry.  The link  $L =LK_\E(A)$  of   $A \in \E$ in~$X_\E$  has
a vertex $v_s$ for every glide $s\in G$ such that $sA\in \E$. A set of
vertices $\{v_{s}\}_s$ spans a simplex in $L $
 whenever   $s$ runs over a cubic set of glides in~$G$ (cf.  the proof of Lemma~\ref{simple}; here we use the square condition on~$\E$).
The link, $K$,  of
$\ast  $ in $Y$  has
two vertices $w^+_s$ and $w^-_s$ for every glide $s \in G $. A set of
vertices $\{w^{\pm}_{s}\}_s$ spans a simplex in $K$
 whenever   $s$ runs over a pre-cubic set of glides.   The map $f_A: L \to K$ induced by $f$ carries  $v_s$ to $w_{\vert s
\vert}^{\varepsilon_s}$. This map   is an embedding since we can recover~$s $ from
${\vert s
\vert}$ and ${\varepsilon_s}$. The square condition  on~$\E$ implies that if the images of two  vertices of~$L $
under $f_A$ are adjacent in~$K $, then the vertices themselves are  adjacent in~$L $.
Since~$L $ is a  flag simplicial
complex,   $f_A(L )$ is a full subcomplex of~$K$.
\end{proof}

\begin{corol}\label{simpleNEW++-nilp}  If the set of glides in $G$ is finite and an orientable   regular set $\E\subset G$  satisfies  the square condition,
 then for all $A\in \E$,  the glide group   $\pi_1(X_\E, A)$ is  biorderable, residually nilpotent, residually
finite, and
 embeds in $SL_n(\ZZ)$ for some~$n$.
\end{corol}

\begin{proof}
 Finitely generated right-angled Artin groups  have  all  the properties listed in   this corollary, see \cite{DT}, \cite{DJ}, \cite{CW},  \cite{HsW}. These properties are
hereditary and  so  are shared by all subgroups of finitely
generated right-angled Artin groups.
Combining with Theorem~\ref{simpleNEW++} we obtain the  desired result.
\end{proof}

\begin{corol}\label{simpleNEW++-eee}  If  the gliding system in $G$ is regular, the number
		 of pairwise independent glides in $G$ is bounded from above,  and $G$ is oriented in the sense of Section~\ref{prelimArtin++}, then        $\mu_A:\pi_1(X_G, A)\to \A $ is an injection for all $A\in G$.
\end{corol}

\begin{proof} This follows from Theorem \ref{simpleNEW++} because  the square condition on $G$ is void. \end{proof}

  \section{The dimer complex and the dimer group}\label{dimergroups}

\subsection{Cycles in graphs}\label{Cycles in graphs}   By a {\it graph} we   mean a non-empty 1-dimensional
CW-complex  without isolated 0-cells (i.e., 0-cells not incident to any 1-cells) and without loops (i.e., 1-cells with equal endpoints).   The 0-cells and     1-cells of a graph are called
 {\it vertices} and  {\it edges}, respectively. We allow multiple edges with the same endpoints.
 A {\it subgraph} of a graph~$\Gamma$ is a graph~$\Gamma'$ embedded in~$\Gamma$ such that all vertices/edges of~$\Gamma'$ are  also vertices/edges of~$\Gamma$.

 Given a set~$s$ of edges  of a graph  $\Gamma $, we denote by $\partial s$ the set of vertices of~$\Gamma$ adjacent to at least one edge in~$s$. Such vertices of~$\Gamma$ are called   \emph{vertices of~$s$}. The set~$s$  is {\it cyclic}
  if it  is   finite and   each vertex of~$s$      is incident to precisely two  edges in~$s$.   The  vertices of   a cyclic set~$s$ together with   the  edges in  $s$   form a subgraph of~$\Gamma$ denoted   $\underline s$ and homeomorphic to a disjoint union of a finite number of circles.   A cyclic set of edges~$s$ is a {\it cycle} if~$\underline s$ is a single circle.
 A cycle  is {\it even} (respectively, {\it odd})  if   it   includes an even (respectively, {\it odd})  number of edges of~$\Gamma$.  Any even cycle   has a unique partition into two   subsets  called the {\it halves}  such that the edges belonging to the same half have no common vertices.

\subsection{The even-cycle gliding system}\label{theglidingsystemcycles} Let   ${\Gamma}$ be a graph with   the set of edges~$E$, and let $G=G(\Gamma)=2^E$ be the power group   of $E$. Two  sets   $s,
t  \subset E$ are {\it independent} if     the edges belonging to $s$ have no
common vertices with the edges  belonging to~$t$. Such sets $s,t$ are necessarily disjoint.

\begin{lemma}\label{glidi}  The even cycles   in ${\Gamma}$ in the role of glides together with the
independence relation   above form   a
  regular   gliding system in   $G$.
\end{lemma}

    All axioms of a gliding system are straightforward.
The    regularity follows from
Lemma~\ref{glidi-}. We call the resulting gliding system in $G$ the \emph{even-cycle} gliding system. By Section~\ref{cubecomplexX}, it
determines a cubed complex $ X_G $ with $0$-skeleton~$G $.
By Corollary~\ref{simpleNEWBIS},   $ X_G $  is
nonpositively curved.





We show how to orient  $ G$   in the sense of
Section~\ref{prelimArtin++}.
  Pick an element $e_s\in s$ in every even cycle $s\subset E$.
 A 1-cell of $X_G$
relates two 0-cells $A,B\subset E$ such that $
AB=(A\setminus B)\cup (B\setminus A)\subset E$ is an even cycle.  Then
  $e_{AB} \in AB$ belongs either to~$A$ or to~$B$. We orient this 1-cell   towards
  the 0-cell containing $e_{AB}$. It is easy to see that this procedure defines an orientation on
  $G$.  By Section~\ref{prelimArtin++}, the latter determines a typing  homomorphism $ \pi_1(X_G, A)\to \A(G) $ for   $A\in  G$.

All cycles in~$\Gamma$ (even and odd)  with the   independence relation   above also form a regular gliding system and yield a  nonpositively curved cubed complex. Some of our results extend to this gliding system but we will not study it.

\subsection{Perfect matchings}\label{newDimer coverings}   Let   ${\Gamma}$ be a graph with   the set of edges~$E$.  We provide the power group~$G=2^E$
with the even-cycle gliding system. A {\it perfect matching}, or a {\it dimer covering}, on~${\Gamma}$ is a
subset   of   $E$  such that every vertex of ${\Gamma}$ is incident to
exactly one edge  in this subset.
 Let   ${\D}=\D({\Gamma})\subset G $
be the set (possibly, empty) of all perfect matchings on~${\Gamma}$.
 By
Section~\ref{cubecomplexX+}, this set    determines a
cubed complex $X_{\D} =X_{\D}  (\Gamma) \subset X_G$ with
0-skeleton~${\D}$. We call $X_{\D}$ the {\it dimer complex} of
${\Gamma}$. By definition, two perfect matchings $A, B\subset E$ are connected by an edge in $X_\D$ if and only if $ AB \subset E$ is an even cycle.   Note   that if $ AB$ is a cycle then this cycle is    even
with   halves $ A\setminus B$ and $ B\setminus  A$.

\begin{lemma}\label{dimersMM} The set of perfect matchings
 ${\D}\subset G$  satisfies  the cube condition of   Section \ref{cubecomplexX+} and the square condition     of Section \ref{incl++}.
\end{lemma}

\begin{proof}
The cube condition follows from the square condition. The latter says
that for any $A\in {\D}$ and any independent even cycles $s,t$ in
${\Gamma}$ such that $sA, tA\in {\D}$, we must  have $stA\in
{\D}$.
The inclusions $A, sA \in {\D}$ imply that    $s\cap A$ and $s\setminus A$ are the   halves of $s$. The inclusions $A, tA \in {\D}$ imply that  $t\cap A$ and $t\setminus A$ are the   halves of $t$. The independence of $s, t$ ensures that
  $s,t$ are
disjoint and   incident to disjoints sets of vertices.
The set   $stA\subset E$ is obtained from   $A $ through  simultaneous
replacement of the half $s\cap A$ of $s$ and the half $t\cap A$ of $t$ with the complementary halves.   It is clear   that $stA$ is a dimer
covering.
\end{proof}

 \begin{theor}\label{sspeccaseW} The dimer complex $X_{\D}$  is
nonpositively curved.
\end{theor}

This theorem follows from Lemma~\ref{dimersMM} and Corollary
\ref{simpleNEWBIS}.

We next determine when two perfect matchings $A,B $ in~$\Gamma$ belong to the same connected component of $X_\D$. We say that $A, B$ are {\it congruent} if the set
  $AB=(A \setminus B) \cup (B\setminus
A)$ is  finite.

 \begin{lemma}\label{conn---} Two perfect matchings in~$\Gamma$ belong to the same connected component of $X_\D$ if and only if they are  congruent.
\end{lemma}

\begin{proof}  Pick any   $A,B \in \D \subset X_\D$. If $A, B$ belong to the same connected component of $X_\D$, then $A$ may be connected to $B$ by a sequence of edges in $X_\D$. In other words, $A$ can be obtained from  $B$ by a finite sequence of glidings along even  cycles. Each such gliding removes a finite set of edges from the   matching and adds another finite set of edges. Hence,
  $AB=(A \setminus B) \cup (B\setminus
A)$ is a finite set. Conversely, suppose that
 the set   $AB $ is  finite. The definition of a perfect matching implies  that $AB$ is cyclic. Any finite cyclic set of edges in $\Gamma$ splits uniquely      as  a   union of    independent   cycles. Let $AB=\cup_{i=1}^n s_i$ be such a splitting with $n\geq 0$.   All the  cycles $s_1, \ldots , s_n$   are even: their halves are their intersections   with $A \setminus B$ and   $B \setminus A$. Then $$A, \, s_1A,  \, s_{2} s_1A,  \,   \ldots ,  \, s_n\cdots s_2s_1 A=B$$ is a sequence of perfect matchings in which every term is obtain from the previous one by gliding along an even cycle. Hence $A$ is connected to $B$ by a sequence of edges in $X_\D$.
\end{proof}

 For a  perfect matching $A$ of~$\Gamma$, the group $\pi_1(X_\D, A)$ is called the \emph{dimer group of~$\Gamma$ at $A$} or, shorter, the \emph{dimer group of~$A$}.
 Lemma~\ref{conn---}  implies that the  dimer groups of congruent perfect matchings are isomorphic.
We will see in Section~\ref{redu} that  the isomorphism  in question
  may be chosen in a canonical way. More precisely, for any congruent perfect matchings $A, B$ in~$\Gamma$ we will define a canonical isomorphism ${i}_{A,B}:\pi_1(X_\D, A) \to \pi_1(X_\D, B)$ and for  any  $A, B, C\in \D$ we will define  an element $\overline{ABCA} $  of  $\pi_1(X_\D, A) $ such that

(i) ${i}_{A,A}=\id$ and ${i}_{B,A}={i}_{A,B}^{-1}$ for any $A, B\in \D$;

(ii)    the automorphism ${i}_{C,A} {i}_{B,C}{i}_{A, B}  $  of  $\pi_1(X_\D, A) $ is the conjugation by $\overline{ABCA} $ for any   $A, B, C\in \D$.

\subsection{The case of a finite graph}\label{thecaseoffinitegraphs} A graph is   \emph{finite} if it has a finite number of vertices and edges. Let ${\Gamma}$ be a finite graph with the set of edges~$E$ and let $G=2^E$ be the power group of~$E$
with the even-cycle gliding system. The   cubed complex~$X_G$    is a finite CW-space and so is compact and  finite dimensional. Since $X_G$ is non-positively curved, it is aspherical.  For  $A\in  G$, the group $\pi_1(X_G, A)$ shares    the properties of the fundamental groups of compact finite dimensional nonpositively curved cubed complexes listed in Section~\ref{Cubed complexes}. Also,   the right-angled Artin   group $\A(G) $ is    finitely generated, and   the typing   homomorphism $ \pi_1(X_G, A)\to \A (G)$ determined by any orientation of~$G$ is   injective  (Corollary \ref{simpleNEW++-eee}). Therefore $\pi_1(X_G, A)$ shares     the properties of   finitely generated right-angled Artin  groups listed in
  Corollary~\ref{simpleNEW++-nilp}. The same arguments yield similar statements   for the dimer complex $X_\D$ of~$\Gamma$  and the dimer groups of perfect matchings in~$\Gamma$.
By Corollary~\ref{simpleNEW++-},    the inclusion   homomorphism $ \pi_1(X_\D, A)\to \pi_1(X_G, A)$ is injective for all  $A\in \D=\D(\Gamma)$.

 \begin{lemma}\label{conn} The dimer complex  of a finite graph  is path connected.
\end{lemma}

This lemma is a direct consequence of Lemma~\ref{conn---} and the obvious fact that all perfect matchings of a finite graph are congruent.
Lemma~\ref{conn}  implies that the  dimer group $ \pi_1(X_\D, A) $   does not depend on  the choice of $A\in \D$  up to    isomorphism.
This group, considered up to  isomorphism, is called     the \emph{dimer group} of~$\Gamma$ and denoted $D({\Gamma})$. By definition, if $\Gamma$ has no perfect matchings, then $D(\Gamma)=\{1\}$.

\subsection{Examples}\label{rerere} 1. Let ${\Gamma}$ be a  triangle (with 3   vertices and 3
edges). The  set of glides in $G=G(\Gamma)$  is empty, $X_G  =G$ consists of 8 points,  $X_\D  =\D  =\emptyset$, and $D(\Gamma)=\{1\}$.

2. Let ${\Gamma}$ be a  square  (with 4   vertices and 4
edges). Then $G=G(\Gamma)$ has   one glide, $X_G  $ is a disjoint union of 8  closed intervals,  $X_\D  $ is one of them, and $D(\Gamma)=\{1\}$.

3.   More generally, let ${\Gamma}$ be formed by $n\geq 1$ cyclically connected vertices and $\D=\D(\Gamma )$. If $n$ is odd, then     $X_\D =\D =\emptyset$. If $n$ is even, then ${\Gamma}$ has two perfect matchings,      $X_\D $
  is a segment, and $D(\Gamma)=\{1\}$.

4. Let ${\Gamma}$ be  formed by 2 vertices
  and $3 $   connecting them edges. Then $G=G(\Gamma)$ has 3 glides and $X_{G} $ is a disjoint union of two complete graphs on 4 vertices. The space $X_\D $ is formed by 3 vertices and 3 edges of one of these   graphs. Clearly, $D(\Gamma)=\ZZ$.

 5. More generally, for   $n\geq 1$, consider the graph ${\Gamma}^n$ formed by 2 vertices
  and~$n $   connecting them edges. A perfect matching in ${\Gamma}^n$ consists of a single edge, and so, the set $\D=\D(\Gamma^n)$ has $n$ elements. The graph ${\Gamma}^n$ has $n(n-1)/2$
 cycles, all  of length 2 and none of them independent.    The complex $X_{\D} $ is
a complete graph on~$n$ vertices. Hence, $D({\Gamma}^n)$ is a free group of rank $(n-1)(n-2)/2$.


  \subsection{Remarks}\label{rexexexexsere} 1. If a  graph   does not have   perfect matchings, then one can  subdivide some of its edges into two subedges so that the resulting graph has   perfect matchings and the theory above applies.   Subdivision of edges into three or more  subedges is redundant. If
a graph  ${\Gamma}'$   is obtained from a  graph ${\Gamma}$   by adding an even number of new vertices inside
an   edge, then there is a canonical bijection $\D( \Gamma) \approx \D( \Gamma') $ which extends to a cubic homeomorphism  $X_{\D} ( \Gamma)   \approx X_{\D} ( \Gamma') $.

   2. If ${\Gamma}$ is a disjoint union of graphs $ {\Gamma}_1$, $ {\Gamma}_2$, then $G=G(\Gamma)=G_1\times G_2$, where $G_i=G(\Gamma_i)$ for $i=1,2$, and
  $X_G   = X_{G_1}   \times X_{G_2}  $. 
  Similarly, $\D=\D(\Gamma)= \D_1 \times \D_2$, where $\D_i=\D(\Gamma_i)$ for $i=1,2$, and  $X_\D=X_{\D_1}\times X_{\D_2}$.
  If    $\Gamma_1$, $\Gamma_2$ are finite graphs, then $D(\Gamma)=D(\Gamma_1)\times D(\Gamma_2)$.

  3. The previous remark    and Example~\ref{rerere}.4 imply that any free abelian group  of
 finite rank is realizable  as the  dimer group of a   finite graph. Other
 abelian groups   cannot be realized as dimer groups of finite graphs
 because  the latter are finitely generated and torsion-free. It would be interesting to find a finite graph whose dimer group  is not a product of free groups of rank $n(n-1)/2$ with $n \geq 1$.

 4. Consider  finite graphs ${\Gamma}_1$, ${\Gamma}_2$
admitting perfect matchings. Let ${\Gamma}'$ be   obtained from   $\Gamma={\Gamma}_1 \amalg {\Gamma}_2$ by adding an edge
connecting a vertex of ${\Gamma}_1$ with a vertex of ${\Gamma}_2$. Then
 $D({\Gamma}')=D({\Gamma} )$ and $X_\D({\Gamma}')= X_\D({\Gamma} )$.
This implies that for   any  finite graph, there is a connected finite graph with the same dimer group.

\section{A presentation  of the dimer group}\label{redu}

We give   a presentation of the dimer group of a   graph  by generators and relations. We begin with a  general result concerning so-called straight  CW-spaces.

\subsection{Straight CW-spaces}\label{mincu} The cellular structure  of a CW-space   is commonly   used to read  from its 2-skeleton  a    presentation  of    the fundamental group by generators and relations.  We discuss   a different  method   producing a presentation of   the fundamental group for   some CW-spaces.
 We call a  CW-space~$X$ \emph{straight} if   all closed cells of~$X$ are embedded balls and   for
  any 0-cells $A,B \in X$   the intersection of all closed cells of~$X$  containing  $A,B $  is a
closed  cell  of~$X$. This intersection is    called   the \emph{hull} of $A,B$.   Connecting   $A$ and~$B$   by a path in their hull, we obtain a well-defined homotopy class of paths  from~$A$ to~$B$  in $X$  denoted $\overline{AB}$. For a  finite sequence   of 0-cells $A_0, A_1, \ldots, A_m\ $ the product of    $\overline{A_0A_1}, \overline{A_1A_2}, \ldots, \overline{A_{m-1}A_m}$ is a  well-defined homotopy class of paths    from $A_0$ to~$A_m$. It is  denoted  $\overline{A_0 A_1 \ldots A_m}$.

 \begin{lemma}\label{phi-+}    Let   $X$ be a straight CW-space, and let $ Y=X^{(0)}$ be the 0-skeleton of~$X$.  Then $X$ is path-connected and for each  $A_0 \in Y
  $,
   the group $\pi_1(X, A_0)$ is   isomorphic to  the group with generators $\{x_{A,B}\}_{A,B\in Y}$ subject to the following relations: $x_{A_0,A}=1$ for all $A\in Y$ and
  $x_{A,C}=x_{A,B}\,  x_{B,C}$ for each triple   $A,B,C \in Y$ such that there is a closed cell   of $X$   containing  $\{A,B,C\}$.
    \end{lemma}

\begin{proof}  The path-connectedness of $X$ is obvious. Let $\Pi$ be the group defined by the generators
and relations   in the statement. The relation $x_{A_0,A}\,
x_{A,A}=x_{A_0,A}$ implies that $x_{A,A}=1$ for all $A\in Y$. Then
  $x_{A,B}\, x_{B,A}=x_{A,A}=1$   for all $A, B\in Y$ so that $x_{B,A}=x_{A,B}^{-1}$. In particular,
$x_{A,A_0}=x_{A_0,A}^{-1}=1$ for all $A\in Y$.

We define a   homomorphism $\phi:\Pi \to \pi_1(X, A_0)$ by $\phi(x_{A,B})=\overline{A_0 ABA_0}$. This definition is compatible with the relations in $\Pi$. Indeed, for   $A\in Y$,
$$\phi(x_{A_0,A}) =\overline{A_0 A_0 A A_0}=\overline{A_0 A  A_0} =1.$$
If  $A,B,C\in Y  $ lie in a closed  cell  of $X$, then $\overline{  A B C}=\overline{  A C}$ and therefore
 $$\phi(x_{A,B}\,  x_{B,C} )  =\overline{A_0 A B A_0   BC A_0 } =\overline{A_0 A B  C A_0 }
 = \overline{ A_0 A C A_0} =\phi(x_{A, C} ) .$$
 We next define a homomorphism $\psi: \pi_1(X, A_0) \to \Pi$.  Note   that if   $A,B \in Y$ are connected by a (closed) 1-cell~$e$ in~$X$, then $e$ is their hull   and  $\overline{AB}$ is  the homotopy class of~$e$  viewed as a path from $A$ to $B$. Any  $\alpha \in \pi_1(X, A_0)$ may be represented by  a loop in the 1-skeleton $X^{(1)}$ of~$X$   traversing consecutively
  0-cells $A_0, A_1,..., A_m=A_0$.  Then $\alpha = \overline{A_0 A_1 \ldots A_m}$, and we set
  $$\psi(\alpha)= x_{A_0,A_1} \,  x_{A_1,A_2}\,  \cdots  x_{A_{m-1},A_m} \in \Pi.$$
  The right-hand side does not depend on the choice of the loop in $X^{(1)}$ representing~$\alpha$. Any two such loops may be related by the transformations inserting  loops of type $\overline{A B A}$ where $A,B\in Y$ and  loops of type $\overline{B_1\ldots B_m}$   where $B_1, B_2, \ldots,  B_m =B_1 $ are consecutive 0-cells  lying on the boundary of a 2-cell   of $X$. The invariance of $\psi(\alpha)$ under these transformations
   follows  from the equalities $x_{A,B}\, x_{B,A} =1$ and
   $$ \prod_{i=0}^{m-1}  x_{B_i,B_{i+1}}=   1$$
   where  we use that  $B_1, \ldots,  B_m $  lie in the same  closed cell   of $X$. It is clear that
   $$\phi \psi (\alpha)=\phi( \prod_{i=0}^{m-1}  x_{A_i,A_{i+1}} ) = \prod_{i=0}^{m-1}  \overline{A_0 A_i A_{i+1} A_0} = \overline{A_0 A_1 \ldots A_m}=\alpha.$$
   It is easy to deduce from the definitions that $\psi \phi=\id$. Thus, $\phi$ and $\psi$ are mutually inverse isomorphisms.
 \end{proof}

Since a straight  CW space~$X$ is path-connected, the group $\pi_1(X,A)$  does not depend on the choice of $A\in  X^{(0)}$  at least up to isomorphism.  The isomorphisms in question can be chosen in a canonical way as follows. Given   $A,B\in X^{(0)}$, we let
${i}_{A,B}:\pi_1(X , A) \to \pi_1(X , B)$ be the conjugation by $\overline{AB}$:
$${i}_{A,B} (\alpha)=\overline{BA}  \alpha \overline{AB} \in \pi_1(X , B) \quad {\text{for}} \quad \alpha\in \pi_1(X , A) .$$
Clearly,
  ${i}_{A,A}=\id$ and ${i}_{B,A}={i}_{A,B}^{-1}$ for any $A, B\in X^{(0)}$. For any   $A, B, C\in  X^{(0)}$,
   the automorphism ${i}_{C,A} {i}_{B,C}{i}_{A, B}  $  of  $\pi_1(X, A) $ is the conjugation by $\overline{ABCA} $.

\subsection{A presentation of the dimer  group}\label{newDimer coverings+}   Consider a  graph
   ${\Gamma}$ and the  dimer  complex $X_\D$ where  $\D=\D(\Gamma)$ is the set   of   perfect matchings in~$ \Gamma$.

   \begin{lemma}\label{strcubi}    All connected components of  $X_\D$ are straight    CW-spaces.   \end{lemma}

   \begin{proof}  Consider any    $A,B \in \D $   lying in the same connected component     of~$X_\D$. By Lemma~\ref{conn---}, the set
  $AB=(A \setminus B) \cup (B\setminus
A)$  is   finite.  The proof of Lemma~\ref{conn---} shows that $AB$ is a  union of    independent  even cycles in~$\Gamma$. Denote the
set of these cycles by $S$. In the notation of Section~\ref{BU}, $[S]=AB$. As in the proof of Lemma~\ref{dimersMM}, we observe  that
for all $ T\subset S $, the set $[T] A$ is a perfect matching in~$\Gamma$. Recall that the closed cells of~$X_\D$ are embedded cubes. In particular, the based cube $(A,S)$ determines   a cube,~$Q$, in~$X_{\D}$ whose set of vertices contains  both $A$ and  $B=[S]A$. We claim that $Q$ is the hull of $A, B$. To prove this claim, it suffices to show that~$Q$ is a face of any cube $Q'$ in
$X_\D$ whose set of vertices contains   $A$ and  $B$. Such a cube $Q'$  can be represented   by a based cube
$(A,S')$ such that   $B=[T] A$ for some $T\subset S'$. We
have $[T]= BA^{-1} =BA=AB $. Since a finite cyclic set of edges
splits as a   union of independent cycles in a unique way, $T=S$. So,  $Q$ is a face of $Q'$.
 \end{proof}

  Lemma~\ref{strcubi}  allows us to apply  the  constructions  of Section~\ref{mincu} to~$X_\D$.  This gives the    isomorphisms between the dimer groups    formulated  at the end of
Section~\ref{newDimer coverings}.

We will denote the only edge of a perfect matching  $A\in \D$ adjacent to  a vertex~$v$ of~${\Gamma}$  by $A_v$.    A triple $A,B,C\in \D$ is said to be {\it flat} if for any vertex $v$ of ${\Gamma}$, at least two of the edges $A_v, B_v, C_v$ are equal.

\begin{theor}\label{phi-+VBBB}     The dimer group of any   $A_0\in \D$   is   generated by the set $\{x_{A,B}\}_{A,B }$ where $(A,B)$ runs over all ordered pairs of  perfect matchings in~$\Gamma$ congruent to~$A_0$. The defining relations: $x_{A_0,A}=1$ for each  $A\in \D$ congruent to $A_0$ and
  $x_{A,C}=x_{A,B}\,  x_{B,C}$ for each  flat triple   $A,B,C  $ of  perfect matchings   congruent to $A_0$. \end{theor}

\begin{proof}   The claim follows from Lemmass~\ref{conn---},~\ref{phi-+}, and~\ref{strcubi}.  We need only to show that  three  perfect matchings $A,B,C \in \D$ are
   vertices of a cube in $X_\D$ if and only if the triple $A,B,C$ is flat.   Suppose   that the triple $ A,B,C $ is flat.
Then
 every vertex of
   ${\Gamma}$ is incident to a unique edge belonging to  at least  two   of the
sets $A,B,C$. Such edges form a perfect matching, ${K}$,   of~${\Gamma}$.
For any vertex $v$ of ${\Gamma}$,   either   $A_v=B_v=C_v ={K}_v $ or  $A_v=B_v ={K}_v\neq C_v $  up to a permutation of
   $A,B,C$. In the first case,
  the sets $A{K}, B{K},
   C{K}$ contain no edges incident to $v$. In the second case,    $A{K},
   B{K}
 $ contain no edges incident to $v$ while $C{K}$ contains two such edges~$C_v$ and~${K}_v$. This   shows that $A{K}, B{K}, C{K}$ are three pairwise
   independent cyclic sets of edges. They split
 (uniquely) as     unions of independent cycles. All these cycles are even:  their intersections with $K$ are their halves.    Denote the resulting sets of
 even cycles by $X, Y, Z$, respectively. Thus,   $A{K}=[X], B{K}=[Y], C{K}=[Z]$. Equivalently, $A=[X]{K}$, $ B= [Y] {K}$, and $ C=[Z]{K}$. Then the cube in $\D$ determined
 by the based cube $({K}, X \cup Y\cup Z)$ contains $\{A,B,C\}$.

 Conversely, suppose that $A,B,C \in \D$ are
   vertices of a cube in $\D$.  It is easy to see that    there is ${K}\in \D$ and
a   set of independent even  cycles $S $ with a partition $S=X \amalg Y \amalg Z$ such
that $A=[X]{K}$, $B=[Y] {K}$, and $C=[Z]{K}$.
If a   vertex $v$ of~${\Gamma}$ is not a vertex of the cycles in $ S$, then $A_v=B_v=C_v ={K}_v $. Pick any vertex   $v \in \partial s $ with $s\in S$. If $s\in   X$, then $s$ is not a vertex of the cycles in $Y$ and $Z$   and
  $ B_v =C_v =K_v$. The cases $s\in Y$ and $s\in Z$ are  similar. In all cases, at least two of the edges $A_v, B_v, C_v$ are equal. Therefore, the triple $ A,B,C $ is flat.
\end{proof}

\begin{corol}\label{cocolphi-+VBBB}     If~$\Gamma$ is a finite graph, then the dimer group of any   $A_0\in \D=\D(\Gamma)$   is   presented by the generators $\{x_{A,B}\}_{A,B\in \D}$ subject to the  relations $x_{A_0,A}=1$ for each  $A\in \D$ and
  $x_{A,C}=x_{A,B}\,  x_{B,C}$ for each  flat triple   $A,B,C \in \D$. \end{corol}

 \subsection{The fundamental groupoid}  Consider a straight CW-space~$X$  and   the fundamental groupoid $ \pi_1(X, Y)$   where $Y=X^{(0)} $ is the 0-skeleton of $X$.    This   groupoid     is a category whose objects are points  of $Y$ and whose morphisms are homotopy classes of paths in $X$ with endpoints in $Y$.  To  make   composition   compatible with multiplication of paths, we write   composition of morphisms in $ \pi_1(X, Y)$  in the  order opposite to the usual one.
  The  arguments   in the proof of  Lemma~\ref{phi-+} show that  the    groupoid $ \pi_1(X, Y)$
  can be presented by    generators $\{x_{A,B}:A\to B\}_{A,B\in Y}$ subject to the   relations
  $x_{A,C}=x_{A,B}\,  x_{B,C}$ for any    $A,B,C\in Y $ contained in a  cell  of $X$.  The image of   $
x_{A,B}$ in  $\pi_1(X, Y)$  is the morphism $\overline{AB}:A\to B$.
    As a consequence, in the setting of Section~\ref{newDimer coverings+}  the   groupoid $ \pi_1(X_\D, \D)$ is presented by the generators $\{x_{A,B}:A\to B\}_{A,B }$ where $(A,B)$ runs over all ordered pairs of  congruent perfect matchings in~$\Gamma$ and the relations
  $x_{A,C}=x_{A,B}\,  x_{B,C}$ for every  flat triple   $A,B,C  $ of  congruent perfect matchings in~$\Gamma$.

\section{Matching groups and their homomorphisms}\label{gencasem}

We define    matching groups and  study natural  homomorphisms   between them.

\subsection{The matching groups}\label{matchgr--}  Recall that for a  set $A$ of edges of  a graph~$\Gamma$, we      denote by $\partial A$  the set of all vertices of~$\Gamma$ adjacent to at least one  edge in~$A$.  Denote by $\Gamma_A$   the subgraph of~$\Gamma$ whose  vertices are the points of $\partial A$ and whose edges are the  edges of~$\Gamma$  with  both endpoints   in $\partial A$.  In particular, all edges of $\Gamma$ belonging to~$A$ are also  edges of~$\Gamma_A$. This gives a set of  edges of~$\Gamma_A$   denoted $A^p$.

A  matching~$A$  in~$\Gamma$  is a non-empty set of edges of~$\Gamma$  such that different edges in~$A$ have no common vertices. If $A$ is a   matching in~$\Gamma$, then~$A^p$  is a perfect matching in~$\Gamma_A$. The  \emph{matching group}   of~$\Gamma$ at~$A$ is defined to be  the dimer group
of   $\Gamma_A$ at~$A^p$:  $$\pi_A=\pi_A (\Gamma)= \pi_1(X_{\D} (\Gamma_A), A^p).$$
Computing this  group  from the 2-skeleton of $X_\D(\Gamma_A)$ in the standard way, we immediately see that this definition of~$\pi_A$ is equivalent to the one   in the introduction.  If~$A$ is a perfect matching in~$\Gamma$, then $\Gamma_A=\Gamma$ and $A^p=A$ so that the matching group and the dimer group of~$A$ coincide.

The congruence of perfect matchings defined in Section~\ref{newDimer coverings}   extends to arbitrary matchings in~$\Gamma$ as follows.
Two matchings $A,B$  are \emph{congruent} if  the set $AB=(A \setminus B) \cup (B \setminus A)$ is finite and $\partial A=\partial B$. These conditions  imply that $\Gamma_A =\Gamma_B$ and that the perfect matchings $A^p, B^p$ lie in the same connected component of the dimer complex $X_{\D} (\Gamma_A)=X_{\D} (\Gamma_B)$. Since this component  is   straight,     Section~\ref{mincu} yields  an isomorphism
${i}_{A,B}:\pi_A \to \pi_B$. We have
  ${i}_{A,A}=\id$ and ${i}_{B,A}={i}_{A,B}^{-1}$. If $A, B, C$ are congruent  matchings in~$\Gamma$, then
    the automorphism ${i}_{C,A} {i}_{B,C}{i}_{A, B}  $  of  $\pi_A $ is  the conjugation by   $\overline{ABCA} \in \pi_A $.

A matching $A$ in~$\Gamma$ \emph{finite} if
  the graph $\Gamma_A$ is finite  (this condition is stronger than   requiring $A$ to be a finite set of edges). The matching group of a  finite matching  shares all group properties listed  in Section~\ref{Cubed complexes} and
  Corollary~\ref{simpleNEW++-nilp}, cf.\ Section~\ref{thecaseoffinitegraphs}.

    \subsection{Canonical homomorphisms}\label{matchgr} Let $A, A'$ be matchings in a graph~$\Gamma$ such that $A' \subset A$. We construct   a   canonical homomorphism $j_{A', A}: \pi_{A'}\to \pi_A$ so that  $j_{A, A}=\id$ and $ j_{A', A} j_{A'', A'} =j_{A'', A}$ for any sets $A''\subset A'\subset A$.

    Replacing,  if necessary,~$\Gamma$ by $\Gamma_A$ we can assume that $A$ is a perfect matching.   Let:  $E$ be the set of edges of~$\Gamma$; $G=2^E$ be the power group of~$E$;   $\D $   be the set  of perfect matchings  in~$\Gamma$. Let $E'$, $G'$,  and $ \D'$ be  similar objects associated with the subgraph $\Gamma'= \Gamma_{A'}$   of~$\Gamma$.  Clearly,    $E'\subset E$,  $G'\subset G$, every even cycle in $\Gamma'$ is also an even cycle in $\Gamma$, and every cubic set of even cycles in $\Gamma'$ is also a cubic set of even cycles in $\Gamma$.
    We define a map $j:X_{\D'} \to X_\D$ as follows.  Set    $$C=A\setminus A'\subset E. $$  Any  point $b\in X_{\D'}$ is represented by a triple ($B\in \D'$,   a cubic set of
 glides   $S\subset G'$,   $x\in I^S$) such that $[T] B\in \D'$ for all  $T\subset S$. Since   $B\in\D'$, we have  $\partial B=\partial A'$,   $B\cap C=\emptyset$, and    $B\cup C\in \D$.  Let $j(b)  $ be the point of $X_G$ represented by the  triple ($B\cup C $,
   $S $,  $x$). It is easy to check that  $j(b)$ does not depend on the choice of $(B, S, x)$. Moreover, $j(b)\in X_\D\subset X_G$. Indeed, for $T\subset S$, we have $[T] B\in \D'$ and therefore $[T] (B\cup C)=[T] B\cup C\in \D$.  The resulting map $j:X_{\D'}\to X_{\D}$ is  continues and   $j(A')=A $. The induced map in $\pi_1$ is  the homomorphism $j_{A', A}: \pi_{A'}\to \pi_A$.

 The  homomorphisms $\{j_{A', A}\}$ are compatible  with the   homomorphisms $\{i_{A, B}\}$ of Section~\ref{matchgr--} as follows: if $A'\subset A$ and $B'\subset B$ are four matchings in~$\Gamma$ such that $A$ is congruent to $B$ and $A'$ is congruent to $B'$, then the following diagram commutes:
$$
    \xymatrix@R=1cm @C=3cm { \pi_{A'}  \ar[r]^-{i_{A', B'}}  \ar[d]_-{j_{A', A}}
    & \pi_{B'}\ar[d]^-{j_{B', B}} \\
     \pi_A \ar[r]^-{ i_{A, B}}
    & \pi_B\, .
    }
$$

If the graph $\Gamma$ is finite, then the homomorphism $j_{A', A}$ is injective for any matchings $A'\subset A$ in~$\Gamma$. This follows from Lemma~\ref{Wise} and the fact that $j:X'\to X$ is a local isometry; we leave the details to the reader.

\section{Matchings vs. braids}\label{extension-----}

The braid groups of     graphs   share many properties of the matching groups, see \cite{Ab}, \cite{CW}. We   construct      natural     homomorphisms  from the matching groups to the braid groups.



\subsection{The braid groups}\label{confbraid} For a topological space $P$ and an  integer $n\geq 1$, the {\it ordered $n$-configuration space}  $\tilde C_n=\tilde C_n(P) \subset  P^n$ consists of all  $n$-tuples of pairwise distinct points of $P$. The symmetric group $S_n$ acts on $\tilde C_n $ by permutations of the tuples, and the quotient  $C_n =C_n(P)=\tilde C_n /S_n$ is the {\it unordered $n$-configuration space} of~$P$.  For   $c\in C_n  $,  the fundamental group $\pi_1(C_n, c)$ is denoted $B_n(P,c)$ and called the {\it $n$-braid group} of $P$ at $c$.  The covering $\tilde C_n \to  C_n $
determines (up to conjugation) a {\it permutation homomorphism} $\sigma_n: B_n(P,c)\to S_n$.
Given a subspace~$Q$  of~$P$,  the obvious  inclusion $C_n(Q)\subset C_n(P)$  induces {\it the inclusion homomorphism}   $B_n(Q,c) \to B_n(P,c)$ for all $n \geq 1$ and $c\in C_n(Q)$.

The connection between  configuration spaces and  matchings  stems from the following observation. Given a finite matching~$A$ in a graph~$\Gamma$, one can consider the set $\widehat A$ consisting of the mid-points of the edges in~$A$. This set    is a point in the unordered $N$-configuration space $C_N(\Gamma)$  where $N=\card(A)\geq 1$. To define $\widehat A$, we   tacitly assume   that    all   edges of~$\Gamma$ are
  parametrized by   $I=[0,1]$.   This     allows   us to  consider  the mid-points of the edges and, more generally, the convex combinations of  points of
   the edges.   The homomorphisms  from the matching groups to the braid groups defined below   do not depend on the  choice of parametrizations   up to composition with transfer isomorphisms of braid groups along paths in $C_N(\Gamma)$.

\subsection{V-orientations}\label{V-orientations of cycles}  We introduce a notion of a v-orientation of a graph~$\Gamma$ needed in our constructions. Given an even cycle $s $ in~$\Gamma$,  the set of its vertices $\partial s$ has a unique partition into two   subsets  called the {\it v-halves}  such that the vertices of any edge in~$s$  belong to different v-halves.
 A {\it v-orientation} of~$\Gamma$ is a choice of a distinguished v-half in each even cycle in $\Gamma$.


\subsection{The  map $\Theta $}\label{amapcgcgcg} Let $\Gamma$ be a v-oriented finite graph  admitting a  perfect matching. Then $\Gamma$ has $2N$ vertices for some $N\geq 1$ and
  all perfect matchings in~$\Gamma$ consist of~$N$ edges.   The formula $A \mapsto \widehat A$ defines a map from  the set $\D$ of perfect matchings in~$\Gamma$  to   $ C_N(\Gamma)$. We    extend this map to   a continuous map~$\Theta$ from the dimer space $X_\D=X_\D(\Gamma)$ to $ C_N(\Gamma)$.     The extension is based on the following idea: when a perfect matching $A$ is glided along an even cycle $s$, the   points of
  $\widehat A$ not lying on the circle~$\underline s$ do not move while the   points of
  $\widehat A \cap \underline s$ move  to the points of $\widehat {sA}\cap \underline  s$ along~$\underline s$ across the distinguished vertices of~$s$. Here is a   precise definition. A point $a\in X_\D$ is
represented by a triple ($A\in \D$, a finite set $S$ of independent even cycles in~$\Gamma$,    $x\in I^S$) such that $[T]A\in \D$ for all   $T\subset S$.
Replacing, if necessary,~$A$ with $tA$ for some $t\in S$ we can   assume that $x(S)\subset [0,1/2]$. Each edge $e \in A$  not  lying on   $\cup_{s\in S} s$ contributes its mid-point to $\Theta (a)$.  Each edge $e \in A$   lying on a cycle   $s\in S$ has an endpoint,~$a_e $, belonging to  the distinguished  v-half of~$s$ and an endpoint, $b_e $, belonging to  the complementary  v-half of~$s$. The edge~$e$ then contributes to~$\Theta (a)$  the point
$$  (1/2+ x(s)) a_e+ (1/2- x(s)) b_e \in e.$$
The coefficients
are chosen so that  for $x(s)=0$   we get the mid-point of~$e$, and  for $x(s)=1/2$ we get   $a_e$.  When $e$ runs over   $A$, we obtain
an $N$-point set $\Theta (a) \subset \Gamma$.
This gives a continuous map $\Theta: X_\D\to C_N(\Gamma)$ extending the map $\D \to C_N(\Gamma), A\mapsto \widehat A$.

\subsection{The homomorphism $\theta_A$}\label{amghohmommmngcg} Let $A$ be a finite matching in a   v-oriented graph~$\Gamma$ (finite or  infinite).
   The subgraph $\Gamma_A $ of $\Gamma$ defined in Section~\ref{matchgr--}  is finite and has a perfect matching~$A^p$. The v-orientation in $\Gamma$ restricts to a v-orientation in~$\Gamma_A$.  Now, Section~\ref{amapcgcgcg} yields a continuous map   $\Theta: X_{\D}( \Gamma_A) \to C_N( \Gamma_A) $
   where $N=\card(A)\geq 1$.  Let    $\theta_A :  \pi_A \to B_N(\Gamma, \widehat A)$  be the composition   $$ \pi_A=  \pi_A(\Gamma)= \pi_1( X_{\D}(\Gamma_A) , A^p) \stackrel{\Theta_{\ast}}{ \xrightarrow{\hspace*{0.4cm}} }  \pi_1(C_N( \Gamma_A) , \widehat {A^p}) =B_N(\Gamma_A, \widehat {A^p}) \to B_N(\Gamma, \widehat A) $$
   where the right  arrow is the inclusion homomorphism.    The homomorphism~$\theta_A$  is   not necessarily injective  and  may be trivial, see Example~\ref{EEEDDD+}  below.
The composition of    $\theta_A$ with   the permutation homomorphism $\sigma_N: B_N(\Gamma, \widehat A)\to S_N$ can be   computed as follows.   Mark  the edges in  $A $ by the numbers $1,2,..., N$. Any    $ \alpha \in  \pi_A$  is represented  by a sequence of consecutive glidings of $A$ along  certain  even cycles $s_1,..., s_n$. Recursively in  $i=1,..., n$, we accompany the $i$-th  gliding  with the transformation  of the marked  matching which keeps the marked edges   not belonging to~$s_i$ and  replaces each marked edge  in $s_i$ with the adjacent edge in $s_i$ sharing the vertex in the distinguished v-half of $s_i$ and having the same mark. After the $n$-th gliding, we obtain  the   matching~$A$ with a new marking. The resulting permutation of the set $\{1,..., N\}$  is  equal to  $\sigma_N \theta_A (\alpha)$.        Generally speaking,    $\sigma_N \theta_A $ depends on the   v-orientation of~$\Gamma$, see Example~\ref{EEEDDD} below. If~$A$ and~$B$ are      congruent matchings  in~$\Gamma$, then   the   canonical isomorphisms of   matching groups and braid groups at~$A$ and~$B$ conjugate $\theta_A$ and $\theta_B$.

Using   $\theta_A$ one can construct braids   from  matchings in~$\Gamma$:  any  tuple of matchings $A_1,..., A_n$  congruent to $A$
gives rise to the braid $\theta_A(\overline{AA_1\cdots A_nA})\in B_N(\Gamma, \widehat A)$.

\subsection{Example}\label{EEEDDD}  Consider the   graph $\Gamma$  in Figure~\ref{fig1} with   vertices $a,b,c,d,e,f$. This graph has three cycles $s_1$, $s_2$, $s_{12}$ formed, respectively,  by the edges of the left square,   the edges of the right square,   all edges except the middle vertical edge. These cycles are even. We distinguish the  $v$-halves, respectively, $\{a , e\}$, $\{  c,e\}$, $\{b,d,f\}$. The vertical edges of $\Gamma$ form a perfect matching  $A$. The matching group $\pi_A=\pi_1(X_\D (\Gamma), A)$ is an infinite cyclic group with generator $t$ represented by the sequence of glidings $A\mapsto s_1 A \mapsto s_{12} s_1 A \mapsto s_2s_{12} s_1 A$. To compute  $\sigma_3 \theta_A :\pi_A \to S_3$, we  mark the edges in $A$ with   $1$, $2$, $3$ from left to right. The transformations  of~$A$ under the  glidings  are shown in Figure~\ref{fig1}. Therefore $\sigma_3 \theta_A(t)=(231)$.   The opposite choice of the distinguished v-half in  $s_2$ gives $\sigma_3 \theta_A(t)=(213)$.

 \begin{figure}[h,t]
  \includegraphics[width=9cm,height=5.4cm]{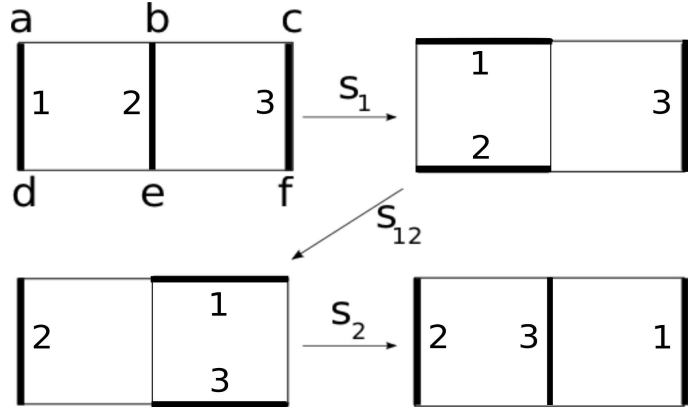}
  \caption{Transformations of a marked perfect matching}\label{fig1}
  \end{figure}

\subsection{Example}\label{EEEDDD+} Consider a     finite  bipartite graph $\Gamma$ (the word \lq\lq bipartite" means  that   the set of  vertices   is partitioned into two subsets $V_0, V_1$ such that every edge has one vertex in each). All cycles in $\Gamma$ are even. We v-orient $\Gamma$   by selecting in every   cycle  the v-half formed by the vertices in $V_0$. For any  matching  $A$ in~$\Gamma$,   we have $\Theta(X_\D(\Gamma_A) )  \subset \prod_{v\in V_0 \cap \partial A} \Delta_v$ where  $\Delta_v $ is the union of all half-edges of $\Gamma_A$ adjacent to~$v$. Since $\Delta_v$ is contractible, the map $\Theta$ is homotopic to a constant map. Hence  $\theta_A =1 $.
Other   v-orientations in $\Gamma$ may give non-trivial~$\theta_A$, cf. Example~\ref{EEEDDD}.

\subsection{Generalization of $\theta_A$}  Let $A$ be a finite matching in a v-oriented graph $\Gamma$ and $N=\card(A)$. The homomorphism $\theta_A :  \pi_A \to B_N(\Gamma, \widehat A)$ can be included into a vast family of similar  homomorphisms  as follows.  Let $F$ be the set of edges of the (finite) graph $\Gamma_A$, i.e., the set of edges of $\Gamma$ with both endpoints in $\partial A$. Pick a  map
  $n: F\to \{0,1,2,...\}$ and set $\vert n\vert =\sum_{e\in F} n(e)$.   Let $\Gamma_A^n$ be the graph obtained from $\Gamma_A$ by adding $2 n(e)$ new vertices inside   each  edge $e  $.
  The  canonical  homeomorphism   $X_\D({\Gamma_A}) \approx  X_\D({\Gamma_A^n})$   (cf. Remark~\ref{rexexexexsere}.1)      induces an isomorphism $\pi_A (\Gamma)  \simeq \pi_{{A_n}}(\Gamma_A^n)$ where ${A_n}$ is the perfect matching in $\Gamma_A^n$ induced by $A^p \in \D(\Gamma_A)$.  The v-orientation of $\Gamma$ induces a v-orientation of $\Gamma^n_A$: the distinguished v-half of a cycle in $\Gamma^n_A$ is the one including   the distinguished v-half of the corresponding cycle in $\Gamma$.   We define a homomorphism $\theta_A^n: \pi_A (\Gamma) \to B_{N+\vert n\vert}(\Gamma, \widehat {A_n})$ as the composition
$$   \pi_A (\Gamma)\simeq \pi_{{A_n}}(\Gamma_A^n) \stackrel{\theta_{{A_n}}}{ \xrightarrow{\hspace*{1cm}} } B_{N+\vert n\vert}(\Gamma_A^n, \widehat {A_n}) \to B_{N+\vert n\vert}(\Gamma,  \widehat {A_n}) $$
where the right arrow is the inclusion homomorphism.
For $n=0$, we have  $\theta_A^n=\theta_A$. In
Example~\ref{EEEDDD+}  the graph $\Gamma^n$ is   bipartite and  so $\theta_A^n=1$   for all $n$.

Composing $\theta^n_A$ with   $\sigma_{N+\vert n\vert} $  we obtain a  family of homomorphisms from the matching groups to the symmetric groups.   Is it true that $\cap\, \cap_n  \Ker (\sigma_{N+\vert n\vert} \theta_A^n )=1$ where the first intersection runs over all   v-orientations of the graph?

  \section{The  space of dimer labelings}\label{lloccgccgcgl}

We interpret the dimer complex of a  finite graph in terms of so-called dimer labelings. This determines a canonical embedding of the dimer complex into a  cube.   We first define   a certain evaluation map   in a more general setting.

\subsection{The evaluation map}\label{orisetthe2}
  Assume   that the power group $G=2^E$ of a set~$E$ is equipped with  a   gliding
system  such that any two independent glides   are disjoint as subsets of~$E$. Consider the associated cubed complex $X_G$. Assigning to each   set $A \subset E$   its
characteristic function $\delta_A: E\to \{0,1\}\subset I$ we obtain
a map  from the 0-skeleton $X^{(0)}_G= G$ of $X_G$ to  $I^E$. This map extends to a map~$\omega: X_G \to I^E$  by
$$\omega  (a)  = \delta_A   + (1-2\delta_A   )  \sum_{s\in S}\,  x(s)\, \delta_{s} :E\to I $$
for any   point $a\in X_G$
represented by a triple ($A\in G$, a cubic set of glides $S \subset G$,  a  map $x:S\to I$).
By the assumption on the gliding system, different elements of~$S$ are
disjoint as subsets of~$E$. Therefore, for any $e\in E$,
\begin{equation}\label{func}
\omega (a) (e) = \left\{
    \begin{array}{ll}
    \delta_A(e) & \mbox{if} \,\, e \in E \setminus  [S]=E \setminus   \cup_{s\in S} s \\
        x(s) & \mbox{if} \,\,   e\in s \setminus A  \,\,  \mbox{with}  \,\,   s\in S\\
        1-x(s) & \mbox{if} \,\,   e\in s \cap A   \,\,  \mbox{with}  \,\,   s\in S
         \, .
    \end{array}
\right.
\end{equation}
 This formula implies that the map $\omega:X_G\to I^E$ is well-defined, continuous, and its restriction    to
any cube in $X_G$ is injective. We call $\omega$ the {\it evaluation map}.


\begin{lemma}\label{Wisecubi}    Suppose that the   gliding system in~$G$ is such that for  any   different cubic sets of glides $S_1,S_2\subset G$, we   have $[S_1]\neq [S_2]  $.  Suppose that   a set $\E \subset G $ satisfies the following condition:
    any glide $s\subset E$ has a partition into two non-empty subsets  such that if $A \in \E$ and $sA \in \E$,
 then $s\cap A$ is one of these subsets.
  Then the restriction of $\omega :X_G\to I^E$  to $X_\E\subset X_G$ is   injective.
   \end{lemma}

   \begin{proof} Assume that
   $\omega (a_1)=\omega (a_2)$ for  some $a_1, a_2\in X_\E$. We  show that $a_1=a_2$. Let us represent each $a_i$ by
  a triple  $(A_i\in \E,S_i,x_i:S_i\to I)$ as above.  Passing, if
   necessary, to  a face of the   cube, we can assume that
   $0<x_i(s)<1$ for all $s\in S_i$.  Set $f_i=\omega(a_i):E\to I$. It follows from   \eqref{func}  that    $  f_i^{-1}( (0,1))=\cup_{s\in S_i} s=[S_i]
   $. The
   equality $f_1 =f_2 $ implies that
   $[S_1]=[S_2]$. By the assumptions of the lemma, $S_1$=$S_2$. Set $S=S_1=S_2$.
We prove below that for
    any  $s \in S$, either

    $(\ast  )$ $s\cap A_1=s\cap A_2$ and $x_1(s)=x_2(s)$ or

    $(\ast \ast)$ $s\cap A_1=s\setminus A_2 $ and
    $x_1(s)+x_2(s)=1$.

    \noindent For  an  $s$ of the second type,    replace $(A_2, S, x_2)$ with    $(A'_2=sA_2, S, x'_2)$  where $x'_2:S\to
   I$ carries $s$ to
   $1-x(s )$ and  is equal to $x_2$ on $S\setminus \{s\}$. The triple $(A'_2, S, x'_2)$ represents the same point
   $a_2 \in X_\E$ and satisfies    $s\cap A_1=s\cap A'_2 $ and
   $x_1(s)=x'_2(s)$. Since different $s\in S$ are  disjoint as
   subsets of $E$, such replacements along different~$s $ of the second type do not interfere with  each other
   and commute. Proceeding in this way, we
 obtain a new triple $(A_2, S, x_2)$ representing
   $a_2 $ and satisfying  $(\ast  )$ for all $s\in S$. Then $x_1=x_2:S\to I$. Also,
   $$A_1\cup \cup_{s\in S} s= f_1^{-1}((0,1]) = f_2^{-1}((0,1]) =A_2\cup \cup_{s\in S} s.$$
   Since different $s\in S$ are disjoint as subsets of $E$ and satisfy $s\cap A_1=s\cap A_2$, we deduce that
    $A_1=A_2$. Thus, $a_1=a_2$.


It remains to prove $(\ast  )$   or $(\ast \ast)$ for
    each   $s \in S$. Set $f=f_1=f_2:E\to I$.
    By \eqref{func},  the
   map $f $   takes at most two  values on
   $s\subset E$.
   Suppose first that $f\vert_s$   takes two distinct values  (whose sum is   equal to $1$).
   By
   \eqref{func},  $f$ is constant on both   $s\cap A_i$   and  $s\setminus A_i$ for   $i=1,2$. If
 $f(s\cap A_1)=f(s\cap A_2)$,  then we have  $(\ast)$. If
 $f(s\cap A_1)= f(s\setminus A_2)$,  then we have  $(\ast\ast)$.   Suppose
   that $f$ takes only one value on $s$. The assumption on the set~$\E$ implies that $s\cap A_1\neq \emptyset$ and $s\setminus A_1\neq \emptyset$.
   Since $f=\omega(a_1)$ takes only one value on $s$, formula \eqref{func} implies that this  value    is  $1/2$.
   Then
$x_1(s)=x_2(s)=1/2$. The assumption on~$\E$ implies  that either $s\cap
A_1=s\cap A_2$ or $s\cap A_1=s\setminus A_2 $. This gives,
respectively, $(\ast)$ or $(\ast \ast)$.
  \end{proof}


 \subsection{Dimer labelings}\label{spaceLgamma} Consider  a finite graph $\Gamma$ with the set of edges $E$.     A {\it dimer labeling} of   ${\Gamma}$ is a
labeling of the edges of   ${\Gamma}$  by non-negative real numbers such
that   for every vertex of   ${\Gamma}$,   only one or two of  the   labels of the adjacent
edges are non-zero and their sum is equal to  $1$.
The set $L= L({\Gamma})$ of dimer labelings  of~${\Gamma}$ is a closed subset of  the cube $I^E$  and is endowed
  with the induced  topology.  The characteristic function of a  perfect matching in~$\Gamma$  is a dimer labeling.   In this way, we  identify the set of perfect matchings $\D= \D(\Gamma)$ of~$\Gamma$  with  a subset of~$L $.


  \begin{theor}\label{spaces++1}

  (i) The set $\D \subset L $ lies  in a path-connected component,  $L_0 $, of   $ L $.

  (ii) Let  $\omega:X_G\to I^E$ be the evaluation map from Section~\ref{orisetthe2} where $G=2^E$ carries the even-cycle gliding system.  The restriction of  $\omega $
   to  the dimer complex $X_\D \subset X_G$   is a homeomorphism of $X_\D $ onto $L_0 $.

   (iii) All  other  components of      $L $ are  homeomorphic to the dimer  complexes of  certain subgraphs of $\Gamma$.
 \end{theor}

 \begin{proof} It follows from the definitions that $\omega(X_\D)  \subset L $ and that the restriction of~$\omega$ to   $  \D \subset X_\D $ is the identity map $\id:\D\to \D$.
We   apply Lemma \ref{Wisecubi} to $ \D\subset G$.
  The condition on the gliding system holds because a  subset of $E$ cannot
split as a   union of independent cycles in two different ways. The partition of  even cycles into halves satisfies the condition on $ \D$.  By Lemma \ref{Wisecubi},  the restriction of~$\omega $ to $X_\D$ is injective.  Since $X_\D$ is path connected, $\omega (X_\D) $ is contained in a path connected component, $L_0$,  of $L(\Gamma) $. By the above,   $\D \subset   \omega (X_\D) \subset L_0$. This proves (i).

We claim that $\omega(X_\D)=L_0$.
Indeed, pick any    dimer labeling $\ell:E\to I $ in $L_0$ and prove that $\ell \in \omega(X_\D)$. It is clear that the set $\ell^{-1}((0,1))\subset E$  is  cyclic. It splits uniquely as a  disjoint  union of  $n \geq 0$ independent  cycles   $s_1,..., s_n$. If $s_i$ is odd for some $i$, then   $\ell \vert_{s_i}= 1/2 $: otherwise, $s_i$ could be partitioned into edges with $\ell<1/2$ and edges with $\ell>1/2$ and would be even. Then, any  deformation of $\ell$ in $L({\Gamma})$   must preserve     $\ell(s_i)$. This contradicts the assumption that $\ell $ can be connected by a path  in~$L$ to an element of $\D$.
Therefore  the cycles $s_1,..., s_n$ are even.  For each $i=1,...,n$, pick a half, $s'_i$, of~$s_i$. The definition of a dimer labeling implies that $\ell$ takes the same value, $x_i\in (0,1)$,  on all edges belonging to $s'_i$ and the value $1-x_i$ on all   edges in  $s_i \setminus s'_i$.    Clearly,  the edges belonging to the set $A=\ell^{-1}(\{1\}) \cup \cup_{i=1}^n s'_i\subset E $  have no common vertices.  Since each vertex of ${\Gamma}$ is incident to an edge with positive label, it is incident to an edge belonging to $A$. Therefore  $A\in \D$.
Since $s_1,..., s_n$ are independent even cycles and $A\cap s_i= s'_i$ for each $i$, all    vertices of the  based cube $(A, S= \{s_1,..., s_n\})$ belong to~${\D}$.  The triple $(A, S, x:S\to I)$, where $x(s_i)=1-x_i$ for all $i$, represents a point $a\in X_\D$ such that  $\omega(a)=\ell$. So,  $\ell \in \omega(X_\D)$.

We conclude  that the map $\omega \vert _{X_\D}: X_\D \to L_0$    is a continuous bijection. Since  $X_\D$ is compact and $L_0  $ is Hausdorff, this map   is a homeomorphism. This proves (ii).

    Arbitrary components of the space $L({\Gamma})$ can be described as follows. Consider a  set ${C}$ of independent odd cycles in~${\Gamma}$. Deleting from~${\Gamma}$ the vertices of   these cycles and   all the edges   incident to these vertices, we obtain a subgraph~${\Gamma}^C$ of~${\Gamma}$. Each   dimer labeling of ${\Gamma}^C$ extends to a dimer labeling of ${\Gamma}$
  assigning $1/2$ to the edges belonging to the cycles in ${C}$ and $0$ to all other edges of~${\Gamma}$ not lying in~${\Gamma}^C$. This defines an embedding $i_C :L({\Gamma}^C)\hookrightarrow L({\Gamma})$.  The   arguments above show  that the image of $i_C  $ is a  union of connected components of $L({\Gamma})$. In particular, $i_C (L_0({\Gamma}^C))$ is a component of $L(\Gamma)$. Moreover,    every component of $L({\Gamma})$ is realized as $i_C (L_0({\Gamma}^C))$ for a unique~${C}$. (In particular,   $L_0({\Gamma})$   corresponds to ${C}=\emptyset$.) It remains to note that $ L_0({\Gamma}^C) $ is homeomorphic to the dimer complex of~$\Gamma^C$.
\end{proof}

\begin{corol}\label{dimlabelB++}   All  connected   components of      $L({\Gamma})$ are aspherical and are homeomorphic to non-positively curved cubed complexes.  \end{corol}


 \section{Extension to hypergraphs}\label{extension}

We extend the   definition of dimer groups and matching groups   to hypergraphs.

\subsection{Hypergraphs}\label{hyper} By a {\it hypergraph} we mean a triple $\Gamma=(E,V, \partial)$ consisting of two  sets $E$, $V$  and a map $\partial:E\to 2^V$ such that $\cup_{e\in E} \, \partial e=V$ and $\partial e\neq \emptyset$ for all $e\in E$. The elements of $E$ are  {\it edges} of $\Gamma$, the elements of $V$ are  {\it vertices} of $\Gamma$, and~$\partial$ is  the {\it boundary map}. For $e\in E$, the elements of $\partial e\subset V$ are   the   {\it vertices of} $e$.

We briefly discuss a few examples. A graph gives rise to  a hypergraph in the obvious way. Every matrix $M$ over an abelian group yields a hypergraph whose edges are   non-zero rows of $M$, whose vertices are   columns of~$M$, and whose boundary map carries a row   to the set of   columns  containing   non-zero entries of this row. A CW-complex gives rise to a sequence of hypergraphs associated as above with the matrices of the boundary homomorphisms in the   cellular chain complex. Every hypergraph $(E,V, \partial)$ determines the {\it dual hypergraph} $(V, E, \partial^*)$ where
$\partial^*(v) =\{e\in E\,\vert\, v\in \partial (e)\}$ for all $v\in V$.


 \subsection{The gliding system}\label{cychyper}  Given a  hypergraph $\Gamma=(E,V, \partial)$, we call two sets $s,t\subset E$ {\it independent} if $\partial s\cap \partial t=\emptyset$. Of course, independent sets are disjoint.

  A {\it cyclic set of edges} in  $\Gamma $
  is a  finite set $s\subset E$ such that for every $v\in V $, the set $\{e\in s\,\vert\, v\in  \partial e\}$   has two elements or is empty. A cyclic set of edges  is a {\it cycle} if
 it does not contain smaller non-empty cyclic sets of edges.

  \begin{lemma}\label{cycliccersMM} If $s\subset E$ is a cyclic set of edges, then the cycles contained in $s$ are pairwise independent and $s$ is their   disjoint union.
\end{lemma}

\begin{proof}   Define a  relation $\sim$ on  $s$  by $e\sim f$ if $e,f\in s$ satisfy $\partial e\cap \partial f\neq \emptyset$. This relation is reflexive and symmetric but possibly non-transitive. It generates an equivalence relation on $s$; the corresponding equivalence classes  are    the cycles contained in $s$.
 This   implies   the lemma.
\end{proof}

   A cycle $s\subset E$ is {\it even} if $s$ has a partition into two   subsets  called the {\it halves}  such that    the edges belonging to the same half have no common vertices.
  It is easy to see  that if such a partition  $s=s' \cup  s''$   exists, then it  is unique and $\cup_{e\in s'} \partial e =\cup_{e\in s''} \partial e $.


Even cycles   as glides  with the
independence relation   above form   a
  regular   gliding system in  the power group $G=2^E$.
By Corollary~\ref{simpleNEWBIS}, the associated cubed complex $ X_G=X_G({\Gamma})$ (with $0$-skeleton $G $)   is
nonpositively curved. As in Section~\ref{theglidingsystemcycles}, a choice of a distinguished element in each even cycle    determines an orientation of $G$ and a typing homomorphism $\mu_A:\pi_1(X_G, A)\to \A(G) $ for   $A\in  G$.
If   $E $ is  finite, then      the right-angled Artin  group    $\A(G) $ is    finitely generated and $\mu_A$ is an injection.

\subsection{Matchings}\label{dimhyper}  A {\it   matching}  in a  hypergraph $\Gamma=(E,V, \partial)$     is a set $A\subset E$ such that $\partial e \cap \partial f=\emptyset$ for any distinct $e,f\in A$.
 A matching $A $ is {\it perfect}
   if  $\cup_{e\in A} \, \partial e=V$.

Let ${\D}=\D(\Gamma) $
be the set   of all perfect matchings of $\Gamma$.  The same arguments as in   the proof of Lemma~\ref{dimersMM}  show that
 ${\D}\subset G=2^E$  satisfies   the square condition and the cube condition.
 The associated
cubed complex $X_{\D}   \subset X_G$     is      the {\it dimer complex} of
$\Gamma$.  Both $X_G$ and $X_\D$ are
nonpositively curved. For $A\in \D$, the group $\pi_1(X_\D, A)$ is the {\it dimer group} of~$\Gamma$ at~$A$. If the   sets   $E$ and~$V$ are finite, then   $X_\D$ is connected;  its  fundamental
group is the {\it dimer group} of~$\Gamma$.

Given a matching $A\subset E$, we define a hypergraph $\Gamma_A=(E_A,V_A, \partial_A)$ where $V_A= \cup_{e\in A} \partial e  $, $E_A=\{e\in E\, \vert \, \partial e\subset V_A\}$, and $\partial_A$ is the restriction of $\partial$ to $E_A$.
Clearly,  $ A\subset E_A$ is a perfect matching in~$\Gamma_A$. The group $\pi_A=\pi_1(X_{\D} (\Gamma_A), A)$ is the {\it matching group} of~$\Gamma$ at~$A$.

  With these definitions, all the results above concerning the dimer groups and the matching groups in graphs extend to
 hypergraphs with appropriate changes. We leave the details to the reader.

\appendix
\section{Typing homomorphisms re-examined}\label{verynewnewDimer coverings++++++++}

 We  state   several    properties of   the  typing homomorphisms defined  in Section~\ref{prelimArtin++}.
Let:  $\Gamma$ be a  graph with the set of edges~$E$;  $G=2^E$ be the power group  of~$G$; and $\D\subset G$ be the set of perfect matchings of~$\Gamma$.
By Section~\ref{theglidingsystemcycles}, a choice of an element $e_s\in s$ in each even cycle~$s$  in~${\Gamma}$  determines an orientation on $G$.  For  $A\in  G$,  consider the corresponding typing  homomorphism $\mu_A:\pi_1(X_G, A)\to \A  $
where $\A=\A(G) $ is the right-angled Artin  group associated with the even-cycle gliding system in~$G$.
 Composing~$\mu_A$
with the inclusion homomorphism $\pi_A =\pi_1(X_{{\D}}, A) \to \pi_1(X_G, A)$ we obtain a homomorphism  $\pi_A \to \A $ also denoted~$\mu_A$. The same homomorphism is obtained by restricting the orientation on~$G$ to~$\D$ and taking the associated typing homomorphism.  If   $E $ is  finite, then $\mu_A:\pi_A \to \A $ is an injection.

If   $e^1_s, e^2_s\in s$  belong to the same half of~$s$ for each even cycle  $s$ in~$\Gamma$, then the families $\{e^1_s\in s\}_s$ and $\{e^2_s\in s\}_s$    determine the same orientation  on $G$ and on~$ \D$. Thus, to specify an orientation  on $\D$ and the typing homomorphism $\mu_A: \pi_A  \to \A$ it suffices to specify a   half in each   $s$. When the distinguished half of an even cycle~$s_0$  is replaced with the complementary half,
the   homomorphism $\mu_A $ is replaced by  its composition with the automorphism of $\A$ inverting the generator $g_{s_0}\in \A$  and fixing  the generators $g_s\in \A$ for  $s\neq s_0$.



Let $\B$ be the  right-angled Artin group
  with generators $\{h_e\}_{e \in E }$ and    relations $h_e h_f=
  h_f h_e$ for all edges $e,f\in E$ having no common vertices. Choose in each even cycle $s$ in $\Gamma$   a half $s'\subset s$ and consider the associated typing homomorphism $\mu_A: \pi_A  \to \A$ with  $A\in \D$. The formula $u(g_s)= \prod_{e\in s\setminus s'} h_e ^{-1} \prod_{e\in s'} h_e  $ defines a homomorphism $u:\A\to \B$. We claim that $u\mu_A=1$. To see this,    use~\eqref{pathee} to compute   $ \mu_A$ on a path~$\alpha$ in~$X_\D$ formed by $n\geq 0$ consecutive 1-cells $e_1,..., e_n$. If $A_k\in \D$ is the terminal endpoint   of
$e_k$  and the initial endpoint of $e_{k+1}$, then $$u(g_{\vert e_k\vert}^{\nu_k})= \prod_{e\in A_{k-1}} h_e ^{-1} \prod_{e\in A_k} h_e.$$
 Multiplying   over  $k=1, \ldots , n$,  we obtain $u\mu_A(\alpha)= \prod_{e\in A_{0}} h_e ^{-1} \prod_{e\in A_n} h_e$. For $ A_0=A_n=A$,   this gives $u\mu_A(\alpha)=1$.

 The  homomorphism $\mu_A : \pi_A \to \A $  induces a  homomorphism in cohomology (with any coefficients) $\mu_A^\ast:H^\ast(\A)\to  H^\ast(\pi_A )$. This may give non-trivial cohomology classes of $\pi_A$. The algebra $H^\ast(\A)$  can be computed from  the    fact that the  cells  of the Salvetti complex  appear in the form of tori and so, the boundary maps in the cellular chain complex are zero    (see, for example,  \cite{Ch}).  The equality $   u \mu_A =1$  above  implies that $\mu_A^\ast$ annihilates   $u^*( H^\ast(\B))\subset   H^\ast(\A)$.


                     \end{document}